\documentclass[a4paper, 11pt]{article}

\usepackage[english]{babel}
\usepackage[OT2,T1]{fontenc}% for Milnor's Lobachevsy function $\ML$
\usepackage[charter]{mathdesign}
\usepackage[%
breaklinks,
colorlinks=true,
linkcolor=black,
anchorcolor=black,
citecolor=black,
filecolor=black,
menucolor=black,
urlcolor=black]{hyperref}
\usepackage{mathtools}
\usepackage{amsmath}
\usepackage{amsthm} 
\usepackage[textwidth=360pt,textheight=648pt]{geometry}
\usepackage{graphicx}
\usepackage{pinlabel}
\usepackage{pgf}
\usepackage[alwaysadjust]{paralist}

\newcommand{\R}{\mathbb{R}}
\newcommand{\Rbar}{\overline{\R}}
\newcommand{\Z}{\mathbb{Z}}
\newcommand{\lob}{\mbox{\fontencoding{OT2}\fontfamily{wncyr}\fontseries{m}\fontshape{n}\selectfont L}}

\newcommand{\Teich}{\mathcal{T}}
\newcommand{\decTeich}{\widetilde{\Teich}}
\newcommand{\Tm}{Teich\-m{\"u}l\-ler}

\newcommand{\Deltil}{\widetilde{\Delta}}
\newcommand{\Deltilo}{\Deltil^{\circ}}
\newcommand{\Delbar}{\bar{\Delta}}
\newcommand{\lamtil}{\tilde{\lambda}}
\newcommand{\lambar}{\bar{\lambda}}
\newcommand{\elltil}{\tilde{\ell}}
\newcommand{\Thetatil}{\widetilde{\Theta}}
\newcommand{\Del}{\operatorname{\textit{Del}}}
\newcommand{\Acal}{\mathcal{A}}
\newcommand{\Hcal}{\mathcal{H}}
\newcommand{\Ecal}{\mathcal{E}}
\newcommand{\Ecalbar}{\bar{\Ecal}}
\newcommand{\Hsf}{\scalebox{0.94}{$\mathsf{H}$}}
\newcommand{\Esf}{\scalebox{0.94}{$\mathsf{E}$}}
\newcommand{\Vtet}{V_{\mathit{tet}}}
\newcommand{\Vo}{V^{\circ}}

\newcommand{\ubar}{\bar{u}}
\newcommand{\xbar}{\bar{x}}
\newcommand{\etal}{{et al.}}
\newcommand{\dtil}{\tilde{d}}
\newcommand{\RP}{\R\mathrm{P}}

\theoremstyle{plain}
\newtheorem{theorem}{Theorem}[section]
\newtheorem{proposition}[theorem]{Proposition}
\newtheorem{lemma}[theorem]{Lemma}
\newtheorem{corollary}[theorem]{Corollary}
\newtheorem{defprop}[theorem]{Definition and Proposition}

\theoremstyle{definition}
\newtheorem{problem}[theorem]{Problem}
\newtheorem{definition}[theorem]{Definition}
\newtheorem{remark}[theorem]{Remark}
\newtheorem{remarks}[theorem]{Remarks}

\title{Ideal hyperbolic polyhedra and\\ discrete uniformization}
\author{Boris Springborn}
\date{}

\begin{document}

\maketitle

\begin{abstract}
  We provide a constructive, variational proof of Rivin's realization
  theorem for ideal hyperbolic polyhedra with prescribed intrinsic
  metric, which is equivalent to a discrete uniformization theorem for
  spheres. The same variational method is also used to prove a
  discrete uniformization theorem of Gu~\etal\ and a corresponding
  polyhedral realization result of Fillastre. The variational
  principles involve twice continuously differentiable functions on
  the decorated \Tm{} spaces $\decTeich_{g,n}$ of punctured surfaces,
  which are analytic in each Penner cell, convex on each fiber over
  $\Teich_{g,n}$, and invariant under the action of the mapping class
  group.
  
  \vspace{\baselineskip}\noindent%
  57M50, 52B10, 52C26
\end{abstract}

%\showthe\textwidth
%\showthe\textheight

\section{Introduction}

This article is concerned with two types of problems that are in fact
equivalent: realization problems for ideal hyperbolic polyhedra with
prescribed intrinsic metric, and discrete uniformization problems. We
develop a variational method to prove the respective existence and
uniqueness theorems. Special attention is paid to the case of genus
zero, because it turns out to be the most difficult one. 
In particular, we provide a constructive variational proof of Rivin's
realization theorem for convex ideal polyhedra with prescribed
intrinsic metric:

\begin{theorem}[Rivin~\cite{rivin94:_intrin}]
  \label{thm:rivin}
  Every complete hyperbolic surface $S$ of finite area that is
  homeomorphic to a punctured sphere can be realized as a convex ideal
  polyhedron in three-dimensional hyperbolic space $H^{3}$. The
  realization is unique up to isometries of $H^{3}$.
\end{theorem}

The realizing polyhedron is allowed to degenerate to a \emph{two-sided
  ideal polygon}. The uniqueness statement of Theorem~\ref{thm:rivin}
implies that this is the case if and only if $S$ admits an orientation
reversing isometry mapping each cusp to itself.

An analogous realization result for convex euclidean polyhedra was
proved by Alexandrov~\cite[pp.~99--100]{alexandrov05:_convex}, and
Rivin's original proof of Theorem~\ref{thm:rivin} follows the general
approach introduced by Alexandrov: First, show that the realization is
unique if it exists. Then use this rigidity result to show that the
space of realizable metrics is open and closed in the connected space
of all metrics. This topological argument does not provide a method of
actually constructing a polyhedron with prescribed intrinsic metric,
and to find such a method was posed as a problem for further
research~\cite{rivin94:_intrin}.

The proof of Theorem~\ref{thm:rivin} presented here is variational in nature. It proceeds by
transforming the realization problem into a finite dimensional
nonlinear convex optimization problem with bounds constraints (see
Theorem~\ref{thm:variational}). This optimization problem is then
shown to have an adequately unique solution (see Section~\ref{sec:proof}). The number of variables is $n-1$ for a
sphere with $n$ cusps. The target function
$\Ecalbar^{v_{\infty}}_{\Delta,\lambda}$ (see Definition~\ref{def:Ecalbar}) is twice continuously
differentiable and piecewise analytic (see
Proposition~\ref{prop:Ecalbar}). The main work of proving the
differentiability statement is done in Section~\ref{sec:difflemma}.

Calculating a value of the target
function involves Epstein and Penner's convex hull
construction~\cite{epstein88,penner87,penner12} for surfaces. For the purposes of
this article, it is necessary to translate this construction into the
language of ideal Delaunay decompositions (see Sections~\ref{sec:delaunay}
and~\ref{sec:akiyoshi}). An ideal Delaunay triangulation
can be found using Weeks's edge flip algorithm~\cite{weeks93}. Once the
Delaunay triangulation is known, the target function and its first and
second derivatives are given by explicit equations (see
Proposition~\ref{prop:Ecalbar}).

A variational proof of Alexandrov's realization theorem for euclidean
polyhedra was given by Bobenko and
Izmestiev~\cite{bobenko08:_alexan_delaun}. Their proof also provides a
constructive method to produce polyhedral realizations, and there are
some similarities between their approach and ours. The variational
principles are analogous, and Delaunay triangulations play an
important role, too. But there is one important difference: The
variational principle of Bobenko and Izmestiev involves a non-convex
target function, while the target function considered here is
convex. This makes the case of ideal polyhedra actually simpler than
the case of euclidean polyhedra, for which no convex variational
principle is known.

In Section~\ref{sec:uniform} we turn to the other side of the theory,
discrete conformal maps. The realization Theorem~\ref{thm:rivin} is
equivalent to a discrete uniformization theorem for spheres,
Theorem~\ref{thm:unisphere}. The equivalence of discrete conformal
mapping problems and realization problems for ideal hyperbolic
polyhedra was established in a previous
article~\cite{bobenko15}. Previously, we treated conformal mapping
problems and polyhedral realization problems with fixed
triangulations. In this article, we require the variable
triangulations to be Delaunay. The previously established variational
principle extends to the setting of variable triangulations. For
discrete conformal maps, this extension can be described roughly as
follows: Minimize the same function as in~\cite{bobenko15}, but flip
to a Delaunay triangulation before evaluating it. However, instead of
using standard euclidean edge flips that do not change the piecewise
euclidean metric of the triangulation, use \emph{Ptolemy flips}:
Update the length of the flipped edge using Ptolemy's
relation~\eqref{eq:ptolemy}. This does change the piecewise euclidean
metric, except if the adjacent triangles are inscribed in the same
circle. Nevertheless, Ptolemy flips are the right thing to do because
they do not change the induced hyperbolic metric (see
Proposition~\ref{prop:dcehyp} and Definition~\ref{def:dce2}).

In Section~\ref{sec:highergenus} we use the variational approach to
prove the uniformization theorem of Gu \etal{}~\cite{luo13} (see
Theorem~\ref{thm:luo13}) and a polyhedral realization result for tori
(see Theorem~\ref{thm:uniformtori}) that was proved by
Fillastre~\cite[Theorem~B]{fillastre08} using Alexandrov's method. Note that
Rivin's Theorem~\ref{thm:rivin} and Theorem~\ref{thm:unisphere} on the
uniformization spheres are not covered by the work of Gu
\etal{}~\cite{luo14,luo13}.

% It may be possible to adapt the variational method presented here to
% other problems, for example discrete conformal mapping problems of
% surfaces with boundary.

A few other cases are known in which a variational principle reduces a
polyhedral realization problem to convex optimization. In the
euclidean setting, Izmestiev~\cite{izmestiev08:cap} observed that the
variational approach to Alexandrov's
theorem~\cite{bobenko08:_alexan_delaun} leads to a convex optimization
problem if one considers the realization of convex caps instead of
convex polyhedra.
%(Alexandrov deduced the theorem for caps from his
%realization theorem for polyhedra~\cite{alexandrov05:_convex}.)

In the hyperbolic setting, a very general realization result for
polyhedral surfaces of arbitrary genus and with finite, ideal or
hyperideal vertices is due to Fillastre~\cite{fillastre08}, who
applied Alexandrov's method, building on Schlen\-ker's
work~\cite{schlenker01} on an even more general result. In principle,
it is possible to derive variational principles in this general
setting from Schl\"afli's differential volume formula for hyperbolic
tetrahedra (see, e.g., \cite[Sec.~5.5]{bobenko15}). But only in some
special cases will this lead to \emph{convex} optimization
problems. This is due to the fact that the volume of a hyperbolic
tetrahedron is a concave function of its dihedral angles only in some
special cases.

A few of the cases allowing polyhedral realization by convex
optimization have already been treated. Izmestiev and
Fillastre~\cite{fillastre09} consider the case of hyperbolic
polyhedral surfaces of genus one with finite vertices. (The volume of
a hyperbolic tetrahedron with one ideal and three finite vertices is
concave.)  Most recently, Prosanov~\cite{prosanov18} has treated
hyperbolic polyhedral surfaces of genus $\geq 2$ with ideal
vertices. (The building blocks are ideal tetrahedra with one
hyperideal and three ideal vertices, truncated at the polar plane of
the hyperideal vertex.)  Prosanov's work also provides a variational
proof of the uniformization theorem involving piecewise hyperbolic
surfaces by Gu et al.~\cite{luo14}. While Prosanov does not provide
explicit formulas for his Hilbert--Einstein functional in terms of
triangulations and lengths, it is straightforward to adapt the known
variational principle for a fixed
triangulation~\cite[Sec.~6]{bobenko15} to obtain such expressions (see
also Remark~\ref{rem:hilbert-einstein}). Such explicit formulas are
useful if one wants to use the variational principle to solve the
realization/uniformization problems numerically.

An interesting sub-case of Fillastre's realization
theorem~\cite{fillastre08} that has not yet been treated by the
variational method is the realization problem for hyperideal polyhedra
with prescribed metric. Since the volume of a hyperideal tetrahedron is
concave~\cite{Schlenker_Rigidity}, this leads to a convex variational
principle. Simultaneously, this would provide a constructive proof of
an existence and uniqueness statement for hyperideal circle
patterns~\cite{BowersBowersPratt}, but with variable
triangulation. Inversive distance circle patterns in the euclidean
plane correspond to polyhedra with hyperideal vertices except for one
ideal vertex. Note that the related realization result for hyperideal
polyhedra with given combinatorial type and \emph{prescribed dihedral
  angles}~\cite{BaoBonahon} has already been treated
variationally~\cite{BobenkoDimitrovSechel,Springborn_Hyperideal}.

Note that even if the variational method is applied to realize
three-dimension\-al polyhedra, the problems are all essentially
two-dimensional. Whether any truly three-dimension\-al problem can be
treated in a similar fashion seems to be one of the most interesting
questions raised by this approach. In particular, ``deformations'' of
two-dimensional problems, like problems involving quasi-Fuchsian
manifolds~\cite{moroianu09}, might have a chance to be tractable.

\section{Overview: Polyhedral realizations from realizable
  coordinates}
\label{sec:overview}

We want to show that the following problem has a unique solution
up to isometries of $H^{3}$:

\begin{problem}
  \label{prob:realize1}
  Given a complete finite area hyperbolic surface $S$ homeomorphic to
  a sphere with $n\geq 3$ punctures, find a realization as convex
  ideal polyhedron.
\end{problem}

We assume the hyperbolic surface $S$ is specified in Penner
coordinates $(\Delta,\lambda)$, consisting of a triangulation $\Delta$
of the sphere with $n$ marked points and a function
$\lambda:E_{\Delta}\rightarrow\R$ on the set $E_{\Delta}$ of edges
(see Section~\ref{sec:penner}). These coordinates determine the
hyperbolic surface $S$ together with an ideal triangulation~$\Delta$
and a choice of horocycle at each cusp.  For each
edge~$e\in E_{\Delta}$, the coordinate~$\lambda_{e}$ is the signed
distance of the horocycles at the ends of $e$ (see
Figure~\ref{fig:signdist}, Remark~\ref{rem:notation}).

\begin{remark}[shear coordinates]
  Rivin~\cite{rivin94:_intrin} assumes that the surface $S$ is
  specified by shear coordinates. This is not an issue because it is
  straightforward to convert between shear coordinates and Penner
  coordinates (see Proposition~\ref{prop:pennershear}).
\end{remark}

The basic idea of our approach is the following: For any polyhedral
realization of $S$ and any choice of a distinguished vertex
$v_{\infty}$, there are special adapted coordinates
$(\Deltil,\lamtil)$ with certain characteristic properties. We call
them \emph{realizable coordinates with distinguished vertex
  $v_{\infty}$} (see Definition~\ref{def:realizable}). The realizable
coordinates $(\Deltil,\lamtil)$ describe the same hyperbolic surface
$S$ as the given coordinates $(\Delta,\lambda)$, but the ideal
triangulation and the choice of horocycles are in general
different. Conversely, if realizable coordinates for~$S$ are known, it
is straightforward to reconstruct the polyhedral realization. Thus,
solving Problem~\ref{prob:realize1} turns out to be equivalent to the
problem of finding realizable coordinates:

\begin{problem}
  \label{prob:realize2}
  Given Penner coordinates $(\Delta,\lambda)$ of a complete finite
  area hyperbolic surface $S$ homeomorphic to a sphere with $n\geq 3$
  punctures and a chosen distinguished cusp $v_{\infty}$, find
  realizable coordinates $(\Deltil,\lamtil)$ with distinguished vertex
  $v_{\infty}$ for the same surface.
\end{problem}

Definition~\ref{def:realizable} characterizes realizable coordinates
in terms of the intrinsic geometry of the surface $S$. This
characterization relies on Epstein and Penner's convex hull
construction~\cite{epstein88,penner87,penner12} for cusped hyperbolic
surfaces decorated with horocycles, and Akiyoshi's~\cite{akiyoshi01}
generalization, allowing partially decorated surfaces: Horocycles may
be missing at some (but not all) cusps. The necessary background will
be reviewed in Sections~\ref{sec:delaunay} and~\ref{sec:akiyoshi} in
the language of ideal Delaunay decompositions.

While the intrinsic characterization of realizable coordinates
requires some preparation, it is more straightforward to explain how a
polyhedral realization gives rise to adapted Penner coordinates
from which the realization can easily be reconstructed. Consider a
convex ideal polyhedron $P$ realizing the hyperbolic surface $S$ in
the half-space model of hyperbolic space (see Figure~\ref{fig:poly}).
\begin{figure}
  \centering
  \includegraphics[width=4in]{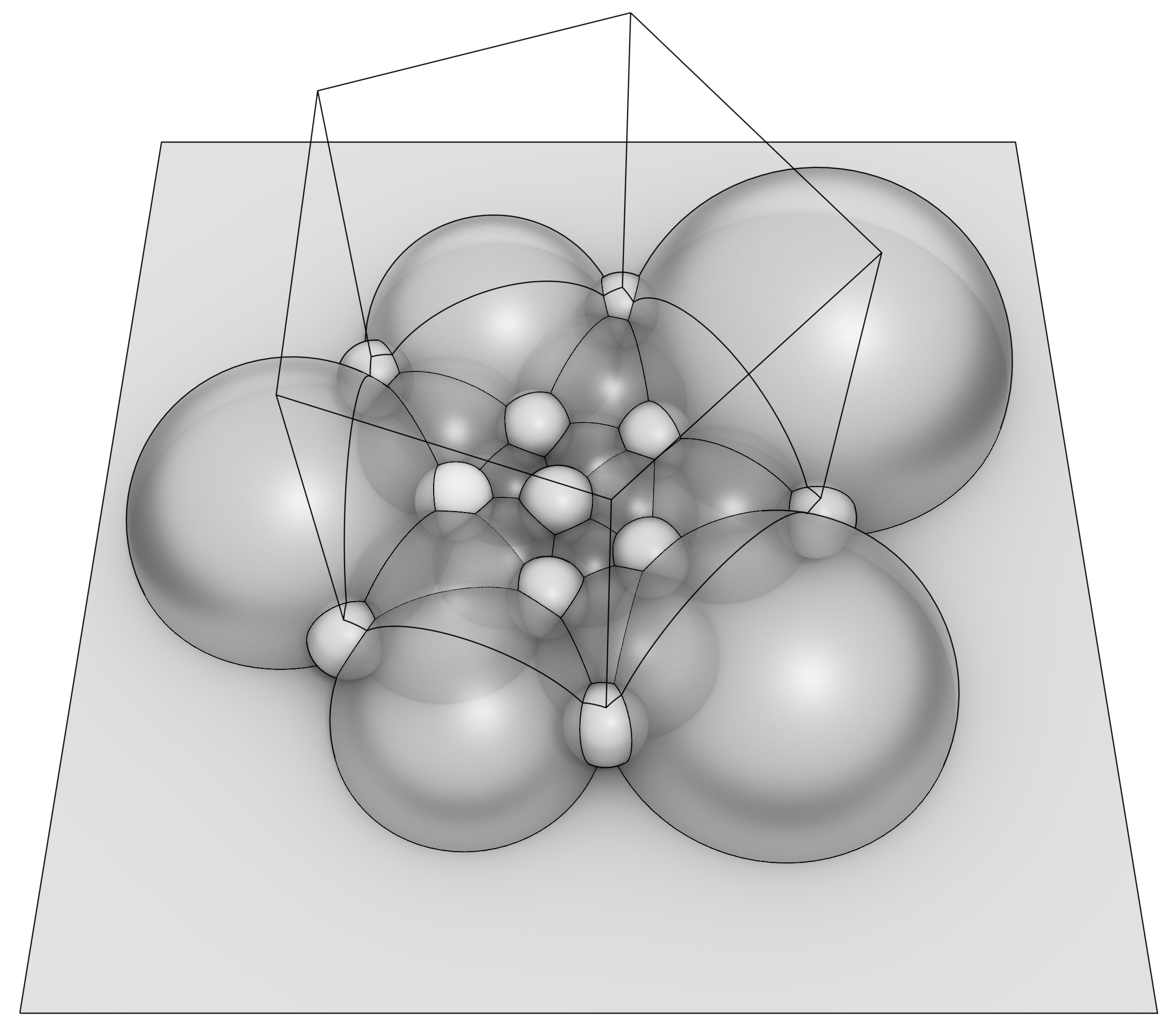}
  \caption{Ideal polyhedron (bounded by the transparent hyperbolic
    planes) decorated with horospheres (white) at ideal vertices}
  \label{fig:poly}
\end{figure}
Assume that one ideal vertex,~$v_{\infty}$, is the point at infinity
in the half-space model. The faces of $P$ are ideal polygons. Should
any face have more than three sides, triangulate it by adding diagonal
ideal arcs. The diagonals may be chosen arbitrarily, except in
vertical faces incident with~$v_{\infty}$, which should be
triangulated by adding the vertical diagonals incident
with~$v_{\infty}$.

For every ideal vertex $v$ of $P$, choose a horosphere $s_{v}$
centered at $v$ as follows: Choose an arbitrary horosphere
$s_{v_{\infty}}$ at $v_{\infty}$ (not shown in
Figure~\ref{fig:poly}). For all other vertices~$v\not=v_{\infty}$
let~$s_{v}$ be the horosphere centered at~$v$ that touches
$s_{v_{\infty}}$ (white spheres in Figure~\ref{fig:poly}).  For each
edge $e$ not incident with~$v_{\infty}$ let $\lamtil_{e}$ be the
signed distance (see Figure~\ref{fig:signdist}) of the horospheres at the
ends of $e$.
\begin{figure}
  \centering
  \labellist
  \small\hair 3pt 
  % \pinlabel {$h_{1}$} [b] at 39 67
  % \pinlabel {$h_{2}$} [b] at 121 57
  % \pinlabel {$h_{1}$} [b] at 196 88
  % \pinlabel {$h_{2}$} [b] at 239 75
  \pinlabel {$\lambda>0$} [b] <0pt,0pt> at 85 42 
  \pinlabel {$\lambda<0$} [b] at 218 21
  \endlabellist
  \centering
  \includegraphics{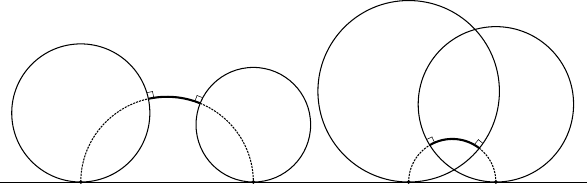}
  \caption{Signed distance $\lambda$ of disjoint and
    intersecting horocycles}
  \label{fig:signdist}
\end{figure}
For each edge $e$ incident with $v_{\infty}$, let $\lamtil_{e}=\infty$.

Intrinsically, the ideal polyhedron~$P$ is the hyperbolic surface
$S$. The triangulated faces of $P$ form an ideal triangulation
$\Deltil$ of $S$. The horospheres $s_{v}$ at the vertices of $P$
intersect the surface $S$ in horocycles $h_{v}=s_{v}\cap P$. The
signed distance $\lamtil_{e}$ of horospheres in $H^{3}$ is also the
intrinsic signed distance of the corresponding horocycles in $S$ along
the ideal arc $e$. By Proposition~\ref{prop:poly2realizable},
$(\Deltil,\lamtil)$ are realizable coordinates with distinguished
vertex $v_{\infty}$ as defined in Definition~\ref{def:realizable}.

To reconstruct the realization $P$ from $(\Deltil,\lamtil)$, proceed
as follows: For each triangle $t\in T_{\Deltil}$ that is not incident
with $v_{\infty}$, construct an ideal tetrahedron as shown in
Figure~\ref{fig:idealtet}.
\begin{figure}
  \labellist
  \small\hair 2pt
  \pinlabel {$\raisebox{10pt}{\Big\uparrow} v_{\infty}$} [ ] at 130 270
  \pinlabel {$v_{1}$} [t] at 69 108
  \pinlabel {$v_{2}$} [t] at 156 77
  \pinlabel {$v_{3}$} [t] at 138 164
  \pinlabel {$\lambda_{3}$} [t] at 110 109
  \pinlabel {$\ell_{3}$} [b] at 115 118
  \pinlabel {$\lambda_{1}$} [r] at 150.5 133.5
  \pinlabel {$\ell_{1}$} [l] <-3pt,7pt> at 154 129
  \pinlabel {$\lambda_{2}$} [tl] at 101 152
  \pinlabel {$\ell_{2}$} [br] at 98 159
  \pinlabel {$\alpha_{1}$} [ ] at 78 135
  \pinlabel {$\alpha_{2}$} [ ] at 151 112
  \pinlabel {$\alpha_{3}$} [ ] at 137 181
  \pinlabel {$A_{1}$} [r] at 66 134
  \pinlabel {$A_{2}$} [l] at 159 104
  \pinlabel {$A_{3}$} [l] at 141 193
  \endlabellist
  \centering
  \includegraphics{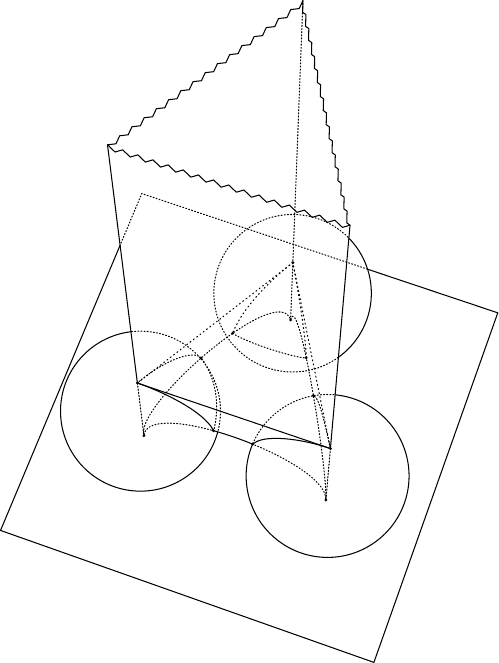}
  \caption{Ideal tetrahedron decorated with horospheres at the
    vertices. The horosphere at $v_{\infty}$ is the horizontal plane at height
    $1$. The horospheres at $v_{1}$, $v_{2}$, $v_{3}$ touch the
    horosphere at $v_{\infty}$. The horosphere at $v_{\infty}$
    intersects the tetrahedron in a euclidean triangle with sides
    $\ell_{i}=e^{\frac{1}{2}\lambda_{i}}$, where $\lambda_{i}$ are
    the signed distances between horospheres (see Figure~\ref{fig:lambdaell}). 
  }
  \label{fig:idealtet}
\end{figure}
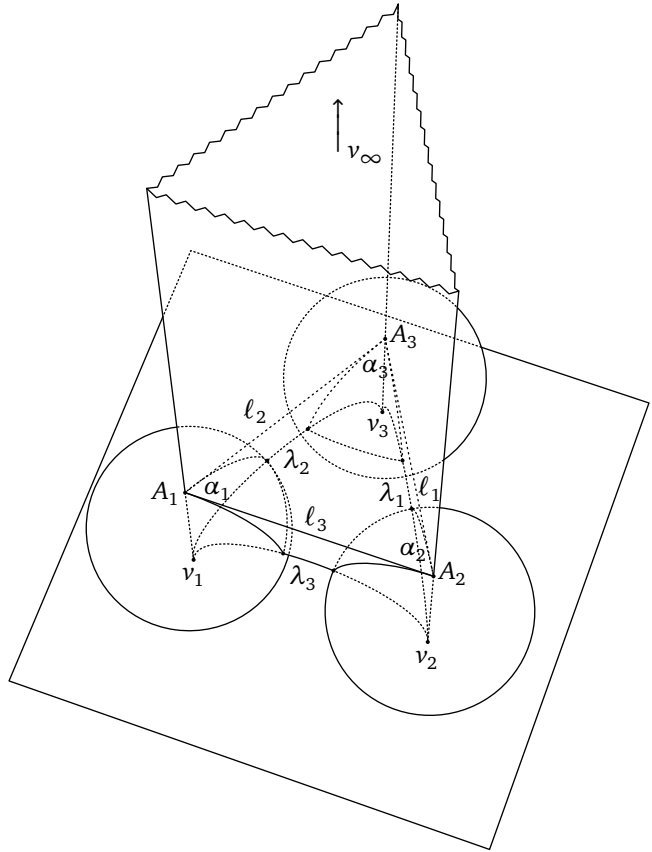
\begin{figure}
  \labellist
  \small\hair 2pt
  \pinlabel {$\ell=e^{\frac{1}{2}\lambda}$} [b] at 63 45
  \pinlabel {$\lambda$} [t] at 63 33
  \endlabellist
  \centering
  \includegraphics{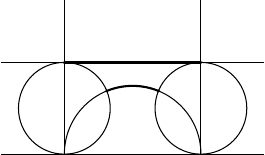}
  \caption{Signed distance $\lambda$ and horocyclic arc length
    $\ell$ in a configuration of two horocycles that are tangent to a
    third horocycle}
  \label{fig:lambdaell}
\end{figure}
These tetrahedra fit together to form the ideal polyhedron~$P$. This
construction of a realizing polyhedron works for any coordinates that
are realizable in the sense of Definition~\ref{def:realizable}
(see Proposition~\ref{prop:realizable2poly}).

The intrinsic Problem~\ref{prob:realize2} of finding realizable
coordinates turns out to be equivalent to a convex optimization
problem (see Theorem~\ref{thm:variational}). This leads to a proof of
Theorem~\ref{thm:rivin} (see Section~\ref{sec:proof}).

\begin{remark}[Hilbert--Einstein functional]
  \label{rem:hilbert-einstein}
  Consider the tetrahedral building block shown in
  Figure~\ref{fig:idealtet}. In this paper, the truncated lengths of
  the base triangles are variable, while the truncated lengths of the
  vertical edges are fixed at zero. Alternatively, we could fix the
  truncated edge lengths of the base triangles and consider the
  truncated lengths of the vertical edges as the variables of the
  optimization problem. This apporach, which was taken
  in~\cite[Sec.~5.4]{bobenko15}, may at first seem more intuitive. It
  also leads to an interpretation of the variational principle in
  terms of a discrete Hilbert--Einstein functional as
  in~\cite{bobenko08:_alexan_delaun,izmestiev08:cap,fillastre09,prosanov18}. Nevertheless,
  we avoid this point of view in this article because it makes it
  harder to understand the role of horocyclic Delaunay triangulations
  (see, e.g., Theorem~\ref{thm:idealeucdel}), which is difficult to
  explain and understand in any case. Moreover, the point of view
  taken here---fixing the vertical lengths at zero and varying the
  base lengths---reduces the three-dimensional realization problem to
  an intrinsic, two-dimensional problem.
\end{remark} 

\section{Penner coordinates}
\label{sec:penner}

In Sections~\ref{sec:penner}--\ref{sec:akiyoshi}, we review known
results from {\Tm} theory. The aim is to fix notation, to collect
required background material and equations for reference, and to
translate the convex hull construction into the language of Delaunay
decompositions. For the reader's convenience, we indicate proofs
whenever we see a way to do so by a short comment or a suggestive
picture. For a more thorough treatment, we refer to the literature. In
this section, we review Penner's coordinates for the decorated {\Tm}
spaces of punctured surfaces~\cite{penner87, penner12}.

Let $S_{g}$ be the oriented surface of genus $g$, let
$V\subseteq S_{g}$ be a finite nonempty subset of $n=|V|$ points and
let $S_{g,n}=S_{g}\setminus V$ be the oriented surface of genus~$g$
with~$n$ punctures. The pair $(S_{g},V)$ is the surface of genus $g$
with $n$ marked points. A triangulation of $(S_{g},V)$ is a
triangulation with vertex set $V$. We will denote the set of edges by
$E_{\Delta}$ and the set of triangles by $T_{\Delta}$. We write
$e_{1}(t)$, $e_{2}(t)$, $e_{3}(t)$ for the edges of a triangle $t$ in
cyclic order, $e_{1}(v),\ldots,e_{\deg(v)}(v)$ for the edges emanating
from a vertex $v$ in cyclic order, and $v_{1}(e)$, $v_{2}(e)$ for the
vertices of an edge $e$ (in arbitrary order).

The \emph{{\Tm} space}~$\Teich_{g,n}$ is the space of equivalence
classes of complete finite area hyperbolic metrics on $S_{g,n}$, where
two such metrics are considered equivalent if they are related by an
automorphism of $S_{g,n}$ that is isotopic to the identity. Loosely
speaking, $\Teich_{g,n}$ is the space of hyperbolic surfaces of genus
$g$ with $n$ cusps. A \emph{decorated hyperbolic surface with cusps}
is a hyperbolic surface with cusps together with a horocycle at each
cusp. The \emph{decorated {\Tm} space}~$\decTeich_{g,n}$ is the space
of decorated hyperbolic surfaces of genus $g$ with $n$ cusps. The
decorated {\Tm} space~$\decTeich_{g,n}$ is a trivial $\R^{n}$-bundle
over the ordinary {\Tm} space $\Teich_{g,n}$.

\emph{Penner coordinates} $(\Delta,\lambda)$ on the decorated
{\Tm} space~$\decTeich_{g,n}$ consist of a
triangulation~$\Delta$ of $(S_{g},V)$ and a function
$\lambda\in\R^{E_{\Delta}}$. The Penner coordinates describe a
complete hyperbolic surface with finite area of genus $g$ with $n$
cusps, marked by $S_{g,n}$, and decorated with a horocycle at each cusp,
together with an ideal triangulation. The cusps correspond to the
vertices $v\in V$ and the ideal triangulation corresponds to a
triangulation in the isotopy class of $\Delta$.  For simplicity, we
will identify cusps with vertices and the ideal triangulation with
$\Delta$.  The value $\lambda_{e}=\lambda(e)$ for an edge
$e\in E_{\Delta}$ is the signed distance (see Figure~\ref{fig:signdist})
of the horocycles at its ideal vertices $v_{1}(e)$, $v_{2}(e)$ as
measured in a lift to the universal cover $H^{2}$.

For two triangulations $\Delta$, $\Deltil$, we denote the \emph{chart
  transition function} 
by
\begin{equation}
  \label{eq:tau}
  \tau_{\Delta,\Deltil}:\R^{E_{\Delta}}\rightarrow\R^{E_{\Deltil}}.
\end{equation}
That is, the Penner coordinates $(\Delta,\lambda)$ and
$(\Deltil,\lamtil)$ describe the same surface if and only if
$\lamtil=\tau_{\Delta,\Deltil}(\lambda)$.

\begin{remark}[notation warning]
  \label{rem:notation}
  Our $\lambda$s denote hyperbolic lengths, not
  Penner's \emph{lambda-lengths}~\cite{penner12}. The lambda-lengths
  are $e^{\lambda/2}$, and we denote them by $\ell$ in this paper. In
  earlier articles~\cite{penner87}, the lambda-lengths were defined as
  $\sqrt{2}\,e^{\lambda/2}$.
\end{remark}

The decorated {\Tm} space $\decTeich_{g,n}$ is a fiber bundle over the
ordinary {\Tm} space $\Teich_{g,n}$. The projection map
$\pi:\decTeich_{g,n}\rightarrow\Teich_{g,n}$ simply forgets about the
decoration. The fibers, whose points correspond to choices of
decorating horocycles, are naturally affine spaces. If a decorated
surface with Penner coordinates~$(\Delta,\lambda)$ is chosen as the
origin in its fiber, then there is a natural parametrization of the
fiber by $\R^{V}$:

\begin{proposition}[parametrizing a fiber of $\decTeich_{g,n}$]
  \label{prop:fiber}
  Let $\Lambda^{(\Delta,\lambda)}$ be the function
  \begin{equation*}
    \Lambda^{(\Delta,\lambda)}:
    \R^{V}\longrightarrow\R^{E_{\Delta}},
    \qquad
    u\longmapsto \Lambda^{(\Delta,\lambda)}(u)
  \end{equation*}
  where the value of $\Lambda^{(\Delta,\lambda)}(u)$ for
  $e\in E_{\Delta}$ is
  \begin{equation}
    \label{eq:Lambda}
    \Lambda^{(\Delta,\lambda)}_{e}(u)
    =\lambda_{e}+u_{v_{1}(e)}+u_{v_{2}(e)}.
  \end{equation}
  Then the decorated surfaces in the fiber of the surface with
  coordinates $(\Delta,\lambda)$ are precisely the surfaces with
  Penner coordinates $(\Delta,\Lambda^{\Delta,\lambda}(u))$ for some
  $u\in\R^{V}$, and~$u$ is uniquely determined by $(\Delta,\lambda)$
  and the decorating horocycles. The horocycle of the decorated
  surface $(\Delta,\Lambda^{\Delta,\lambda}(u))$ at a vertex~$v$ is a
  distance~$u_{v}$ away from the horocycle of $(\Delta,\lambda)$ at
  $v$, measured in the direction of the cusp.
\end{proposition}

The following two propositions provide relations between $\lambda$s
and the lengths of horocyclic arcs.

\begin{proposition}[horocyclic arcs in a decorated triangle]
  \label{prop:arc}
  Consider a decorated ideal triangle with signed horocycle distances
  $\lambda_{1}$, $\lambda_{2}$, $\lambda_{3}$ as shown in
  Figure~\ref{fig:ideal_triang}.
  \begin{figure}
    \labellist
    \small\hair 2pt
    \pinlabel {$\lambda_{1}$} [t] at 64 45
    \pinlabel {$\lambda_{2}$} [bl] at 83 63
    \pinlabel {$\lambda_{3}$} [br] at 44 63
    \pinlabel {$\alpha_{1}$} [t] at 65.5 78
    \pinlabel {$\alpha_{2}$} [bl] at 34 44
    \pinlabel {$\alpha_{3}$} [br] <-0.5pt, -1.5pt> at 96 44.5
%    \pinlabel {$h_{2}$} [t] <1pt,0pt> at 15 51
%    \pinlabel {$h_{3}$} [bl] <0pt, -2pt> at 98 28
%    \pinlabel {$h_{1}$} [r] at 86 102
    \endlabellist
    \hfill
    \includegraphics{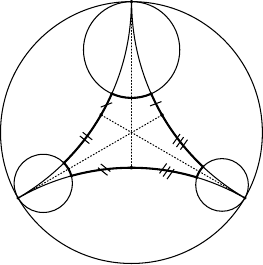}
    \labellist
    \small\hair 2pt
    \pinlabel {$\lambda_{1}$} [bl] <-0.5pt,-0.5pt> at 88 52
    \pinlabel {$\lambda_{2}$} [l] at 131 85
    \pinlabel {$\lambda_{3}$} [r] at 44 85
    \pinlabel {$\alpha_{1}$} [t] at 90.5 117
    \pinlabel {$\alpha_{2}$} [bl] <-2pt,-1pt> at 56 52.5
    \pinlabel {$\alpha_{3}$} [br] <1.5pt, 1.5pt> at 126 44
    \pinlabel {$0$} [t] at 44 9
    \pinlabel {$1$} [t] at 131 9
    \pinlabel {$i$} [br] at 42 92
    \pinlabel {$1+i$} [bl] at 132 92
    \endlabellist
    \hfill
    \includegraphics{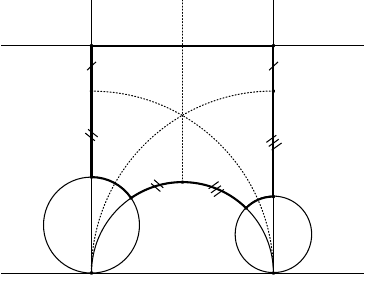}
    \hspace*{\fill}
    \caption{Decorated ideal triangle in the Poincar{\'e} disk model (left) and
      in the half-plane model (right)}
    \label{fig:ideal_triang}
  \end{figure}
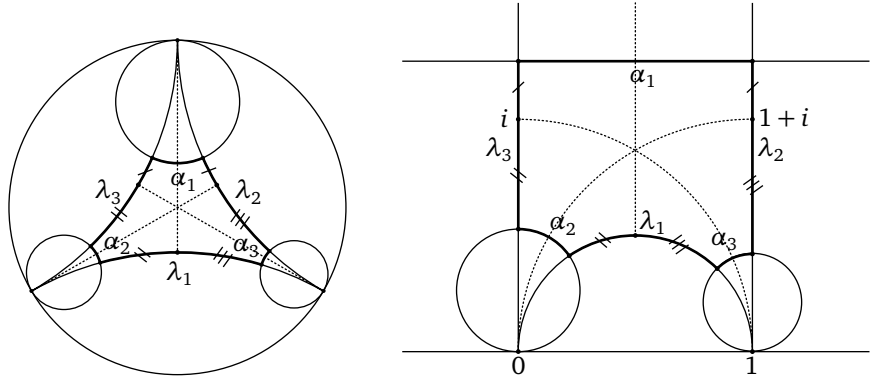
  Each horocycle intersects the triangle in an arc of length
  $\alpha_{k}$. The signed horocycle distances $\lambda$ determine the
  horocyclic arc lengths $\alpha$, and vice versa, via the relations
  \begin{align}
    \label{eq:alpha}
    \alpha_{i}&=e^{\frac{1}{2}(\lambda_{i}-\lambda_{j}-\lambda_{k})},\\
    \label{eq:lambdafromalpha}
    \lambda_{i}&=-\log\alpha_{j}-\log\alpha_{k},
  \end{align}
  where $(i,j,k)$ is a permutation of $(1,2,3)$.
\end{proposition}

Summing the horocyclic arc lengths around one vertex, one obtains the
total horocycle length at a cusp:

\begin{proposition}[horocycle length at a cusp]
  The total length of the horocycle at a cusp $v\in V$ of the
  decorated hyperbolic surface with Penner coordinates
  $(\Delta,\lambda)$ is 
  \begin{equation}
    \label{eq:c}
    c_{v}(\Delta,\lambda)=\sum_{i=0}^{d-1}
    e^{\frac{1}{2}
      \left(
        \lambda_{\widehat{e}_{i}(v)}-\lambda_{e_{i}(v)}-\lambda_{e_{i+1}(v)}
      \right)},
  \end{equation}
  where $e_{1}(v),\ldots, e_{d}(v)=e_{0}(v)$ are the edges emanating
  from $v$ in cyclic order and
  $\hat{e}_{1}(v),\ldots, \hat{e}_{d}(v)=\hat{e}_{0}(v)$ are the edges
  opposite $v$ in cyclic order, so that the $i$-th triangle around $v$
  looks like this:
  \begin{center}
    \begin{pgfpicture}
      \pgfsetlinewidth{0.5pt}
      \begin{pgfscope}
        \pgfpathmoveto{\pgfpointorigin}
        \pgfpathlineto{\pgfpoint{30pt}{15pt}}
        \pgfpathlineto{\pgfpoint{30pt}{-15pt}} 
        \pgfpathclose
        \pgfusepath{stroke,clip} 
        \pgfpathcircle{\pgforigin}{10pt}
        \pgfusepath{stroke}
      \end{pgfscope}
      \pgfpathcircle{\pgforigin}{1.5pt}
      \pgfusepath{fill,stroke}
      \pgftext[x=-5pt]{$v$}
%      \pgftext[x=16pt,y=-1pt]{\small$\alpha_{i}$}
      \pgftext[x=38pt]{$\hat{e}_{i}$}
      \pgftext[x=12pt,y=15pt]{$e_{i+1}$}
      \pgftext[x=13pt,y=-14pt]{$e_{i}$}
    \end{pgfpicture}
  \end{center}
\end{proposition}

\emph{Shear coordinates} $(\Delta,\sigma)$ on the {\Tm} space
$\Teich_{g,n}$ also consist of a triangulation $\Delta$ of $(S_{g},V)$
and a function $\sigma\in\R^{E_{\Delta}}$. Shear coordinates describe
a marked surface together with an ideal triangulation but without
decorating horocycles. For each edge $e\in E_{\Delta}$,
$\sigma_{e}=\sigma(e)$ is the shear with which the ideal triangles are
glued together along $e$ (see Figure~\ref{fig:shear}).
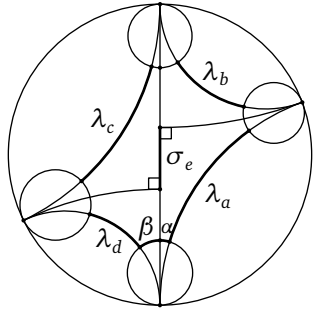
\begin{figure}
\labellist
\small\hair 2pt
 \pinlabel {$\lambda_{a}$} [tl] at 72 48
 \pinlabel {$\lambda_{b}$} [bl] at 71 82
 \pinlabel {$\lambda_{c}$} [br] at 42 64
 \pinlabel {$\lambda_{d}$} [tr] at 45 31
 \pinlabel {$\sigma_{e}$} [l] at 57 55
 \pinlabel {\scriptsize$\alpha$} [b] at 59.5 26
 \pinlabel {$\beta$} [b] at 52 23
\endlabellist
\centering
\includegraphics[scale=1.0]{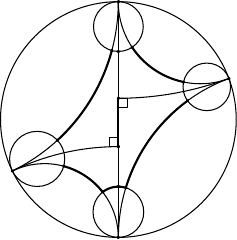}
\caption{Penner coordinates and shear coordinates}
\label{fig:shear}
\end{figure}
The shear coordinates of a complete hyperbolic surface sum to zero
around every cusp:
\begin{equation}
  \label{eq:sigmasum}
  \sum_{i=1}^{\deg(v)} \sigma_{e_{i}(v)}=0.
\end{equation}

\begin{proposition}[Penner coordinates and shear coordinates]
  \label{prop:pennershear}
  (i) The shear coordinates $(\Delta,\sigma)$ of a decorated surface
  with Penner coordinates $(\Delta,\lambda)$ are
  \begin{equation}
    \label{eq:shearfrompenner}
    \sigma_{e}=\frac{1}{2}(\lambda_{a}-\lambda_{b}+\lambda_{c}-\lambda_{d})
    =\log\beta-\log\alpha,
  \end{equation}
  where $a$, $b$, $c$, $d$ are the edges adjacent to edge $e$, and
  $\alpha$, $\beta$ are the horocycle arc lengths, as shown
  in Figure~\ref{fig:shear}.
  
  (ii) Conversely, if $(\Delta,\sigma)$ are the shear coordinates of a
  complete finite area hyperbolic surface $S$, then one obtains Penner
  coordinates $(\Delta,\lambda)$ for $S$, decorated with some choice
  of horocycles, as follows. First, note that the shear coordinates
  determine the length ratios of adjacent horocycle arcs $\alpha$,
  $\beta$ as shown in Figure~\ref{fig:shear} by
  equation~\eqref{eq:shearfrompenner}.  Use the
  relations~\eqref{eq:shearfrompenner} to determine compatible arc
  lengths around each vertex. (The
  equations~\eqref{eq:shearfrompenner} are compatible
  by~\eqref{eq:sigmasum}, but under-determined: At each vertex,
  exactly one arc length may be chosen arbitrarily, and this choice
  fixes the decorating horocycle.) Then use
  equations~\eqref{eq:lambdafromalpha} to determine $\lambda$.
\end{proposition}

The Penner coordinates of a decorated surface with respect to
triangulations $\Delta$ and $\Delta'$ that differ by a single edge
flip are related by Ptolemy's relation~\eqref{eq:ptolemy}:
\begin{figure}
  \begin{minipage}[t]{0.5\linewidth}
    \labellist
    \small\hair 2pt 
    \pinlabel {$a$} [t] at 66 46 
    \pinlabel {$b$} [l] at 107 59 
    \pinlabel {$c$} [bl] at 94 78 
    \pinlabel {$d$} [br] at 45 66 
    \pinlabel {$e$} [br] at 71 60 
    \pinlabel {$f$} [tr] <0pt,4pt> at 89 55 
    \pinlabel {\footnotesize$\alpha_{1}$} [t] at 63 80 
    \pinlabel {\footnotesize$\alpha_{2}$} [tl] at 75.5 84
    \endlabellist
    \centering
    \includegraphics{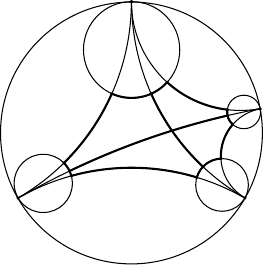}
    \caption{Ptolemy relation}
    \label{fig:ptolemy}
  \end{minipage}
  \begin{minipage}[t]{0.5\linewidth}
    \labellist
    \small\hair 2pt 
    \pinlabel {$\alpha$} [ ] at 42 69
    \pinlabel {$\alpha'$} [ ] at 110 64
    \pinlabel {$\beta$} [ ] at 69 40
    \pinlabel {$\beta'$} [ ] at 81 41.5
    \pinlabel {$\gamma$} [ ] at 57.5 86
    \pinlabel {$\gamma'$} [ ] at 70 87.5
    \pinlabel {$e$} [l] at 69 64
    \pinlabel {$v_{a}$} [r] at 2 74
    \pinlabel {$v_{b}$} [t] at 88 4
    \pinlabel {$v_{a'}$} [l] at 126 62
    \pinlabel {$v_{c}$} [b] at 63 126
    \endlabellist
    \centering
    \includegraphics{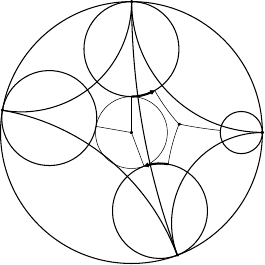}
    \caption{Local Delaunay condition}
    \label{fig:localdel}
  \end{minipage}
\end{figure}

\begin{proposition}[Ptolemy relation]
  \label{prop:ptolemy}
  Consider a decorated ideal quadrilateral with edges $a$, $b$, $c$,
  $d$ and diagonals $e$, $f$ as shown in Figure~\ref{fig:ptolemy}.
  Let $\lambda_{a},\ldots,\lambda_{f}$ be the respective signed
  horocycle distances. Then $\ell_{a},\ldots,\ell_{f}$ defined
  by 
  \begin{equation}
    \label{eq:ell}
    \ell=e^{\frac{1}{2}\lambda}
  \end{equation}
  satisfy Ptolemy's relation
  \begin{equation}
    \label{eq:ptolemy}
    \ell_{e}\ell_{f}=\ell_{a}\ell_{c}+\ell_{b}\ell_{d}.
  \end{equation}
\end{proposition}
\begin{proof}
  This follows with~\eqref{eq:alpha} from
  \begin{equation*}
    e^{\frac{1}{2}(\lambda_{a}-\lambda_{d}-\lambda_{f})}
    +e^{\frac{1}{2}(\lambda_{b}-\lambda_{c}-\lambda_{f})}
    =\alpha_{1}+\alpha_{2}
    =e^{\frac{1}{2}(\lambda_{e}-\lambda_{c}-\lambda_{d})}
  \end{equation*}
  by multiplying both sides with
  $e^{\frac{1}{2}(\lambda_{c}+\lambda_{d}+\lambda_{f})}$.
\end{proof}

\section{Ideal Delaunay triangulations}
\label{sec:delaunay}

Epstein and Penner's \emph{convex hull
  construction}~\cite{epstein88,penner87,penner12} is a fundamental
tool for the polyhedral realization method presented here. The
construction works for cusped hyperbolic manifolds of arbitrary
dimension, but some aspects are simpler for surfaces (compare, e.g.,
the local condition~\eqref{eq:local2} with the generalized tilt
formula~\cite{sakuma95}), and other aspects (like the edge flip
algorithm~\cite{weeks93}) have no adequate counterpart in higher
dimensions. In this section, we will review the relevant results
focusing on the two-dimensional case.

Given a finite area hyperbolic surface with cusps, decorated with a
horocycle at each cusp, the construction of Epstein and Penner returns
an ideal cell decomposition of the given surface, called the
\emph{canonical cell decomposition} of the decorated surface. Roughly,
the construction works as follows: Consider the universal cover of the
given decorated surface in the hyperboloid model of the hyperbolic
plane. Each decorated cusp is represented by an orbit of points in the
positive light cone in a $(2+1)$-dimensional Lorentz space. The convex
hull of these points projects to an ideal cell decomposition of the
surface. We refer to the original articles for a detailed
exposition~\cite{epstein88,penner87,penner12}.

In order to apply the convex hull construction to the polyhedral
realization problems at hand, we need to translate it into the
language of Delaunay decompositions. Accordingly we will use the term
``ideal Delaunay decomposition'' (Definition~\ref{def:idealdel})
instead of ``canonical cell decomposition'', but these terms are
synonymous. The situation is analogous to the classical case of power
diagrams and weight\-ed Delaunay decompositions in euclidean $n$-space,
which can be obtained by a convex hull construction in
$(n+1)$-dimensional space~\cite{aurenhammer13}.

For Epstein and Penner's convex hull construction, the translation
into the language of Delaunay decompositions is a straightforward
application of the pole-polar relationship of projective geometry and
Proposition~\ref{prop:orthotouch} below. It works as follows. In the
hyperboloid model of the hyperbolic plane, circles are intersections
of the hyperboloid with affine planes. Planes intersecting the
hyperboloid correspond, via the pole-polar relationship, to points
outside the hyperboloid. Two circles intersect orthogonally if the
pole of one circle's plane lies in the other circle's plane. The
vertices of the Epstein--Penner convex hull correspond to horocycles,
while the faces correspond to orthogonally intersecting circles. (It
may be necessary to scale the convex hull up to make all faces
intersect the hyperboloid.)  Convexity translates to the condition
that a face-circle intersects all horocycles corresponding to
non-incident vertices less than orthogonally. So far, everything is
analogous to the standard theory of power diagrams and weighted
Delaunay triangulations in euclidean space.
Proposition~\ref{prop:orthotouch} allows us to express the Delaunay
condition in terms of oriented contact instead of orthogonality. This
only works for Delaunay triangulations in hyperbolic space and with
horospheres as sites.

\begin{proposition}[orthogonality and contact]
  \label{prop:orthotouch}
  Let $h_{1},\ldots,h_{n}$ be horocycles with different centers in the
  hyperbolic plane. Of the following statements, any two labeled with
  the same letter are equivalent. If $n\geq 3$, statements labeled
  with different letters are mutually exclusive.
  \begin{compactitem}[$(b_2)$]
    \medskip
  \item[$(a_1)$] There is a circle that intersects every horocycle~$h_{1},\ldots,h_{n}$ orthogonally.
  \item[$(a_2)$] There is a circle that touches every horocycle~$h_{1},\ldots,h_{n}$
    externally.
    \medskip
  \item[$(b_1)$] There is a horocycle that intersects every horocycle~$h_{1},\ldots,h_{n}$ orthogonally.
  \item[$(b_2)$] There is a horocycle that touches every
    horocycle~$h_{1},\ldots,h_{n}$.
    \medskip
  \item[$(c_1)$] There is a hypercycle that intersects every
    horocycle~$h_{1},\ldots,h_{n}$ orthogonally.
  \item[$(c_2)$] There is a hypercycle or geodesic that touches every
    horocycle~$h_{1},\ldots,h_{n}$ on the same side.
  \end{compactitem}
  (A hypercycle is a complete curve at a constant nonzero distance
  from a geodesic.)
\end{proposition}

The following definitions and theorems summarize Epstein
and Penner's results in the language of Delaunay decompositions.

\begin{definition}[ideal Delaunay decomposition]
  \label{def:idealdel}
  Let $S$ be a complete finite area hyperbolic surface with at least
  one cusp, decorated with a horocycle at each cusp. First, assume
  that the horocycles are small enough so that they bound disjoint
  cusp neighborhoods. An ideal cell decomposition $D$ of $S$ is an
  \emph{ideal Delaunay decomposition}, if its lift~$\hat{D}$ to~$H^{2}$
  and the lifted horocycles satisfy the \emph{global Delaunay
    condition}:
  \begin{compactitem}[(gD)]
  \item For every face $\hat{f}$ of $\hat{D}$ there is a circle that
    touches all lifted horocycles centered at the vertices of
    $\hat{f}$ externally and does not meet any other lifted
    horocycles.
  \end{compactitem}
  If the horocycles are not small enough to bound disjoint cusp
  neighborhoods, the cell decomposition $D$ is an ideal Delaunay
  decomposition if condition~(gD) holds for smaller horocycles at
  equal distance to the original horocycles. The distance is arbitrary
  as long as it is large enough for the new horocycles to bound
  disjoint cusp neighborhoods.
\end{definition}

\begin{theorem}[existence and uniqueness]
  \label{thm:idealdel}
  For each decorated complete finite area hyperbolic surface with at
  least one cusp there exists a unique ideal Delaunay
  decomposition. %\upshape{\cite{epstein88} \cite{penner87} \cite{penner12}}
\end{theorem}

The faces of the ideal Delaunay decomposition are ideal polygons. An
ideal Delaunay triangulation is obtained by adding ideal arcs to
triangulate the non-triangular faces:

\begin{definition}[ideal Delaunay triangulation]
  \label{def:deltriang}
  An ideal triangulation $\Delta$ of a decorated complete finite area
  hyperbolic surface is a called an \emph{ideal Delaunay
    triangulation} if it is a refinement of the ideal Delaunay
  decomposition $D$, that is, if $E_{D}\subseteq E_{\Delta}$. An edge
  $e\in E_{\Delta}$ is an \emph{essential edge} of the Delaunay
  triangulation if $e\in E_{D}$. The edges in
  $E_{\Delta}\setminus E_{D}$ are \emph{nonessential}.
\end{definition}

The decorated \Tm{} spaces $\decTeich_{g,n}$ decompose into cells
consisting of all decorated surfaces with the same ideal Delaunay
decomposition:

\begin{definition}[Penner cell]
  \label{def:pennercell}
  For a triangulation $\Delta$ of $(S_{g},V)$, $|V|=n$, let the
  \emph{Penner cell} $\mathcal{C}(\Delta)$ be the set of decorated
  surfaces in $\decTeich_{g,n}$ for which $\Delta$ is an ideal
  Delaunay triangulation. 
\end{definition}

\begin{theorem}[canonical cell decomposition of $\decTeich_{g,n}$]
  \label{thm:pennerdecomp}
  The Penner cells $\mathcal{C}(\Delta)\subseteq\decTeich_{g,n}$ are
  the top-dimensional closed cells of a cell decomposition of
  $\decTeich_{g,n}$.
\end{theorem}

Like in the standard euclidean theory, the global Delaunay condition
(gD) is equivalent to edge-local conditions:

\begin{theorem}[local Delaunay conditions]
  \label{thm:localdel}
  Let $S$ be a complete finite area hyperbolic surface with at least
  one cusp, decorated with small enough horocycles so they bound
  disjoint cusp neighborhoods. An ideal triangulation $\Delta$ is a
  Delaunay triangulation if and only if for every edge $e$, one and
  hence all of the equivalent conditions (lD1)--(lD3) are satisfied.
  Note that the directions left and right relative to $e$ are only
  defined once an orientation is chosen for $e$, but the truth values
  of conditions~(lD1) and~(lD2) are independent of this choice.
  \begin{compactenum}[(lD1)]
  \item The circle touching the three horocycles of the triangle to
    the left of $e$ and the remaining horocycle of the triangle to the
    right of $e$ are disjoint or externally tangent.
  \item The center of the circle touching the three horocycles of the
    triangle to the left of $e$ is to the left of, or coincides
    with, the center of the circle touching the horocycles of the
    triangle to the right of $e$.
  \item The total length of the horocyclic arcs incident with $e$ is
    not greater than the total length of the horocyclic arcs opposite to
    $e$, as shown in Figure~\ref{fig:localdel}:
  \begin{equation}
    \label{eq:local2}
    \alpha+\alpha'\leq\beta+\beta'+\gamma+\gamma'.
  \end{equation}
  \end{compactenum}
  Moreover, if $\Delta$ is a Delaunay triangulation, then an edge $e$
  is nonessential if and only if tangency holds in (lD1), or
  equivalently, if the circle centers coincide in condition~(lD2), or
  equivalently, if equality holds in~\eqref{eq:local2}.
\end{theorem}

The global Delaunay condition~(gD) obviously implies~(lD1), and it is
not difficult to see that (lD1) and (lD2) are equivalent. To see
that~(lD2) and~(lD3) are equivalent, note that the oriented length of
the thick horocyclic arcs in Figure~\ref{fig:localdel} is
$-\alpha-\alpha'+\beta+\beta'+\gamma+\gamma'$. It remains to show that
the local conditions imply the global condition~(gD);
see~\cite[Theorem~5.1]{penner87} \cite[p.~128]{penner12}.

\begin{theorem}[Weeks's flip alogrithm~\cite{weeks93}]
  \label{thm:flip}
  An ideal Delaunay triangulation of the decorated surface
  $(\Delta,\lambda)$ can be found by the flip algorithm: Iteratively
  flip any edge violating the local condition~\eqref{eq:local2} and
  update $\lambda$ using Ptolemy's relation~\eqref{eq:ptolemy} until
  no such edge remains.
\end{theorem}

\begin{remarks}%[practical aspects of the flip algorithm]
  \label{rem:flip}
  (i) Issues of numerical instability that plague the euclidean flip
  algorithm~\cite{edelsbrunner01,fortune95} are also relevant
  in this setting.

  (ii) No other algorithm for computing ideal Delaunay
  triangulations seems to be known.

  (iii) Because the diameter of the flip-graph is infinite, the number
  of flips necessary to arrive at a Delaunay triangulation, depending
  on an arbitrary initial triangulation, is unbounded. (The only
  exception is the sphere with three punctures, which admits only
  three ideal triangulations.)

  (iv) To analyze the complexity of a variational algorithm to solve
  the realization problem~\ref{prob:realize1}, it would be important
  to bound the number of steps in the flip algorithm under the
  condition that the initial triangulation is a Delaunay triangulation
  for a different choice of horocycles.

  (v) Little seems to be known about the following related
  question, except that the number is finite~\cite{akiyoshi01}: How
  many ideal Delaunay decompositions arise for a fixed hyperbolic
  surface as the decoration varies over all possible choices of
  horocycles? A recent result says that the lattice of Delaunay
  decompositions of a fixed punctured Riemann surface is the face
  lattice of an associated secondary polyhedron~\cite{loewe}.

  (vi) Recently, an analogous flip algorithm for surfaces with real projective
  structure was analyzed by Tillmann and Wong~\cite{tillmannwong:2016}.
\end{remarks}

\begin{definition}
  \label{def:Del}
  Let
  \begin{equation*}
    \Del:(\Delta,\lambda)\longmapsto (\Deltil,\lamtil)
  \end{equation*}
  be a function that maps the Penner coordinates $(\Delta,\lambda)$
  of a decorated surface to the Penner coordinates
  $(\Deltil,\lamtil)$ of the same decorated surface with respect to an 
  ideal Delaunay triangulation $\Deltil$.
\end{definition}

Such a function $\Del$ can be computed using the flip algorithm
(see Theorem~\ref{thm:flip}). The function $\Del$ is not uniquely
determined because an ideal Delaunay triangulation is in general
not unique. We will use $\Del$ in situations in which
it makes no difference which ideal Delaunay triangulation is chosen.

\bigskip%
The last main point in this Section is Theorem~\ref{thm:idealeucdel},
which establishes a one-to-one correspondence between
\begin{compactenum}[(a)]
\item ideal Delaunay triangulations of decorated complete hyperbolic
  surfaces, and
\item Delaunay triangulations of closed piecewise euclidean surfaces
  with cone singularities.
\end{compactenum}
Delaunay triangulations and Voronoi diagrams of piecewise euclidean
surfaces with cone singularities were invented and reinvented many
times in the context of different applications, for example
in~\cite{bobenko07,indermitte01,masur91,thurston98}. It seems that
Theorem~\ref{thm:idealeucdel} ought to be known, too, but we do not have a
reference.

First we need the following observation, which is due to
Penner~\cite[Lemma~5.2]{penner87}:

\begin{lemma}[ideal Delaunay triangulations and triangle inequalities]
  \label{lem:triang}
  Let $\Delta$ be a triangulation of $(S_{g},V)$, let
  $\lambda\in\R^{E_{\Delta}}$, and let $\ell\in\R_{>0}^{E_{\Delta}}$ be defined
  by~\eqref{eq:ell}. If the local Delaunay condition~\eqref{eq:local2}
  is satisfied for every edge of the decorated hyperbolic surface with
  Penner coordinates $(\Delta,\lambda)$, then $\ell$ satisfies the
  triangle inequalities for every triangle $t\in T_{\Delta}$:
  \begin{equation}
    \label{eq:triang}
    \begin{split}
      -\ell_{e_{1}(t)}+\ell_{e_{2}(t)}+\ell_{e_{3}(t)}&>0\\
      \ell_{e_{1}(t)}-\ell_{e_{2}(t)}+\ell_{e_{3}(t)}&>0\\
      \ell_{e_{1}(t)}+\ell_{e_{2}(t)}-\ell_{e_{3}(t)}&>0.
    \end{split}
  \end{equation}
\end{lemma}

\begin{remark}
  This is a global statment. The triangle
  inequalities~\eqref{eq:triang} for any particular triangle $t$ are
  satisfied if the local Delaunay conditions~\eqref{eq:local2} are
  satisfied for \emph{all} edges of the triangulation. It is not
  sufficent to assume the local Delaunay condition for, say, the edges
  of triangle $t$, or in some neighborhood.
\end{remark}

If $\Delta$ is a triangulation of $(S_{g},V)$ and
$\ell\in \R_{>0}^{E_{\Delta}}$ satisfies the triangle
inequalities~\eqref{eq:triang} for every triangle $t\in T_{\Delta}$,
then there is a piecewise euclidean metric on $S_{g}$ that turns every
triangle in $T_{\Delta}$ into a euclidean triangle and every edge
$e\in E_{\Delta}$ into a straight line segment of length~$\ell_{e}$.

\begin{definition}
  \label{def:triangpweucsurf}
  We will refer to the surface $S_{g}$ equipped with the piecewise
  euclidean metric described in the previous paragraph and together
  with the triangulation $\Delta$ as \emph{the triangulated piecewise
    euclidean surface $(\Delta,\ell)$}.
\end{definition}

\begin{theorem}[ideal and euclidean Delaunay triangulations]
  \label{thm:idealeucdel}
  Let $\Delta$ be a triangulation of $(S_{g},V)$, let
  $\lambda\in\R^{E_{\Delta}}$, and let $\ell\in\R_{>0}^{E_{\Delta}}$
  be defined by~\eqref{eq:ell}.

  (i) If $\Delta$ is an ideal Delaunay triangulation of the decorated
  hyperbolic surface with Penner coordinates $(\Delta,\lambda)$, then
  $\Delta$ is also a euclidean Delaunay triangulation of the piecewise
  euclidean surface $(\Delta, \ell)$, which exists by
  Lemma~\ref{lem:triang}.

  (ii) Conversely, if $\ell$ satisfies the triangle inequalities for
  every triangle $t\in T_{\Delta}$, and if $\Delta$ is a Delaunay
  triangulation of the piecewise euclidean surface $(\Delta,\ell)$,
  then $\Delta$ is also an ideal Delaunay triangulation of the
  decorated hyperbolic surface with Penner coordinates
  $(\Delta,\lambda)$.
\end{theorem}

\begin{proof}
  First note that a decorated ideal tetrahedron with horosphere
  distances~$\lambda_{i}$ as shown in Figure~\ref{fig:idealtet} exists
  if and only if the $\ell_{i}$ defined by~\eqref{eq:ell} satisfy the
  triangle inequalities. Now consider two decorated ideal tetrahedra
  as shown in Figure~\ref{fig:idealtet} glued along two matching
  vertical faces. The local Delaunay condition (lD2) for the decorated
  ideal triangles in the base planes is equivalent to the local
  Delaunay condition for the euclidean triangles in which the
  horocycle at~$v_{\infty}$ intersects the tetrahedra. To see this,
  note that the circumcenters of the euclidean triangles project
  vertically down to the highest points of the hyperbolic base planes,
  and that these points are the centers of circles touching the
  incident horocycles.
\end{proof}

\section{Akiyoshi's compactification}
\label{sec:akiyoshi}

Every triangulation $\Delta$ of $(S_{g},V)$ with $n=|V|>0$, $2-2g-n<0$
is the Delaunay decomposition for some decorated complete finite area
hyperbolic metric on $S_{g,n}$. For example, the decorated surface
$(\Delta,\lambda)$ with Penner coordinates $\lambda=0$ has Delaunay
decomposition $\Delta$. There are infinitely many triangulations of
$(S_{g},V)$, unless $g=0$ and $n\leq3$. However, a fixed hyperbolic
surface supports only a finite number of ideal Delaunay
triangulations:

\begin{theorem}[Akiyoshi~\cite{akiyoshi01}]
  \label{thm:akiyoshi}
  Let $S$ be a complete hyperbolic surface of finite area with at
  least one cusp. There are only finitely many ideal triangulations
  $\Delta$ of $S$ such that there exists a decoration of $S$ with
  horocycles such that $\Delta$ is a Delaunay triangulation of the
  decorated surface. 
\end{theorem}

In fact, Akiyoshi proved a more general result, the
generalization of Theorem~\ref{thm:akiyoshi} to hyperbolic manifolds
of arbitrary dimension. A simpler proof for the two-dimensional case
can be found in~\cite{luo13}.

Akiyoshi's proof is based on the observation that the convex hull
construction of Epstein and Penner generalizes naturally to
\emph{partially decorated surfaces}, that is, complete hyperbolic
surfaces of finite area, decorated with horocycles at some, but at
least one, of the cusps. A missing horocycle should be interpreted as
the limit of horocycles vanishing at infinity.

We need to consider partially decorated surfaces, their Delaunay
decompositions, and Penner coordinates in some detail because the
definition of \emph{realizable coordinates}
(Definition~\ref{def:realizable}) is based on this: Roughly,
realizable coordinates are Penner coordinates with respect to a
Delaunay triangulation of a partially decorated punctured sphere with
exactly one horocycle missing (condition (r1) of
Definition~\ref{def:realizable}), that satisfy and additional
condition (condition (r2)). In the remainder of this sections, we
collect the relevant definitions facts. We state the results without
proof because they are either straightforward translations of known
facts into the language of ideal Delaunay decompositions, or they are
relatively simple consequences.

\begin{remark}
  The need to deal with partially decorated surfaces is one reason why
  the case of genus zero is more complicated. It is not necessary to
  consider partially decorated surfaces to treat the higher genus
  cases (Section~\ref{sec:highergenus}). Readers interested only in
  the higher genus cases may safely skip the remainder of this section
  and everything in Section~\ref{sec:variational} that has to do with
  missing horocycles.
\end{remark}

The existence and
uniqueness Theorem~\ref{thm:idealdel} generalizes to partially
decorated surfaces:

\begin{theorem}[existence and uniqueness]
  \label{thm:idealdel2}
  Every partially decorated surface has a unique ideal Delaunay
  decomposition.
\end{theorem}

However, Definition~\ref{def:idealdel} of an ideal Delaunay
decompositions has to be modified quite radically: an ideal Delaunay
decomposition of a partially decorated surface with missing horocycles
is not an ideal cell decomposition (see
Definition~\ref{def:delpart}). A cusp with missing horocycle is
contained in a Delaunay face that is not an ideal polygon, but an ideal
polygon with a cusp in the sense of the following definition:

\begin{definition}
  \label{def:punctured}
  An \emph{ideal polygon with a cusp} is a hyperbolic
  manifold-with-boundary $f$ of finite area whose interior is
  homeomorphic to an open disk with one puncture such that a
  neighborhood of the puncture corresponds to a cusp neighborhood in
  $f$ and the boundary is a union of complete geodesics
  (see Figure~\ref{fig:puncturedface}).
\end{definition}
\begin{figure}
  \centering
  \includegraphics{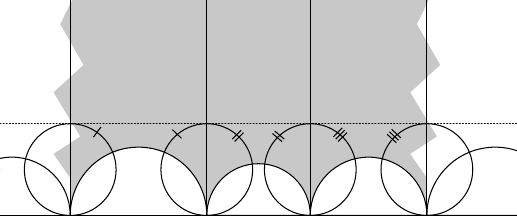}
  \caption{An ideal polygon with a cusp (shaded region), which is a
    face of a Delaunay decomposition of a partially decorated
    surface. The left and right boundary are identified by an isometry
    $z\mapsto z+c$.  The vertical edges are nonessential edges of an
    adjusted Delaunay triangulation.}
  \label{fig:puncturedface}
\end{figure}

However, the Delaunay decomposition of a partially decorated surface
\emph{is} a cell decomposition when viewed as a decomposition of
$(S_{g},V)$, the closed surface of genus $g$ with marked points. The
faces are cells, but the marked points $v\in V$ that correspond to
cusps without decorating horocycle are not vertices of the
decomposition. A face contains at most one marked point in its
interior, which we call the \emph{central vertex} of the face
(although it is not a vertex of the decomposition). Let us call the
regular vertices of an ideal polygon with a cusp its \emph{peripheral
  vertices} to distinguish them from its central vertex.

\begin{definition}[Delaunay decomposition of a partially
  decorated surface]
  \label{def:delpart}
  Let~$D$ be a decomposition of a partially decorated surface into
  ideal polygons and ideal polygons with a cusp, called the
  \emph{faces} of the decomposition. Suppose first that the decorating
  horocycles are disjoint. Then the decomposition $D$ is called a
  \emph{Delaunay decomposition}, if its lift $\hat{D}$ to $H^{2}$ and
  the lifted horocycles satisfy the following conditions:
  \begin{compactitem}[(gD$'$)]
  \item[(gD)] If $\hat{f}$ is the lift of a face of $D$ that is an ideal
    polygon, then there is a circle that touches all lifted horocycles
    at the vertices of $\hat{f}$ externally and does not meet any
    other lifted horocycles. 
  \item[(gD$'$)] If $\hat{f}$ is the lift of a face $f$ of $D$ that is
    an ideal polygon with a cusp, then every peripheral vertex of $f$
    is decorated with a horocycle, and there exists a horocycle
    centered at the lifted central vertex of $\hat{f}$ that touches
    the lifted horocycles at the lifted peripheral vertices of
    $\hat{f}$ and does not meet any other lifted horocycles.
  \end{compactitem}
  Extend this Definition to the case of intersecting horocycles like
  in Definition~\ref{def:idealdel}.
\end{definition}

\begin{remark}
  Generically, the punctured Delaunay faces are punctured
  monogons. Such punctured monogons were already considered by Penner
  to construct coordinates on the \Tm{} space of partially decorated
  punctured surfaces~\cite[Addendum]{penner87} \cite[Sec.~5.1]{penner12}.
\end{remark}

Even if a cusp is not decorated with a horocycle, differences of
distances to horocycles at other cusps are well defined. This can be
used to characterize the punctured Delaunay
faces (see Figure~\ref{fig:puncturedface}): 

\begin{proposition}[characterization of punctured Delaunay faces]
  \label{prop:punctured}
  Let $S$ be a partially decorated surface, let $v$ be a cusp with
  missing horocycle, and let $p$ be an ideal polygon with a cusp in
  $S$ whose central vertex is $v$. Then $p$ is a face of the Delaunay
  decomposition of $S$ if and only if all peripheral vertices of $p$
  are decorated with horocycles, and these horocycles have all the
  same distance to $v$, and this distance is strictly smaller than the
  distance of any other horocycle to $v$.
\end{proposition}

Definition~\ref{def:deltriang} of \emph{ideal Delaunay triangulations}
and \emph{essential} and \emph{nonessential edges} remains valid
without change. However, it will be useful to triangulate
the punctured Delaunay faces in a canonical way:

\begin{definition}[adjusted Delaunay triangulation]
  \label{def:adjusted}
  An ideal triangulation $\Delta$ of a partially decorated surface is
  called an \emph{adjusted Delaunay triangulation} if it refines the
  ideal Delaunay decomposition and every punctured face is
  triangulated by adding the ideal arcs connecting the central vertex
  with the peripheral vertices.
\end{definition}

Theorem~\ref{thm:localdel} on
\emph{local Delaunay conditions} remains valid after the following
modifications: A triangle of a Delaunay triangulation has at most one
vertex with missing horocycle. For an ideal triangle with missing
horocycle at one vertex~$v$, one has to consider the horocycle
centered at $v$ and touching the triangle's two horocycles, instead of
a circle touching three horocycles. If a horocycle is missing, the
respective arc length in condition~\eqref{eq:local2} is zero.

The flip algorithm (see Theorem~\ref{thm:flip}) remains valid, but
some details have to be modified (see
Proposition~\ref{prop:flippartial}) because Penner coordinates
$(\Delta,\lambda)$ do not describe partially decorated
surfaces. Instead, we may use the parametrization of a fiber of
$\decTeich_{g,n}$ described in Proposition~\ref{prop:fiber}, which
extends nicely to partially decorated surfaces:

\begin{definition}[parametrizing an extended fiber of
  $\decTeich_{g,n}$]
  \label{def:extendedfiber}
  Let $(\Delta,\lambda)$ be Penner coordinates for a decorated surface
  in $\decTeich_{g,n}$, let
  \begin{equation}
    \label{eq:Rbar}
    \Rbar=\R\cup\{+\infty\}, 
  \end{equation}
  and
  \begin{equation*}
    u\in\Rbar^{V}\setminus\{+\infty_{V}\},
  \end{equation*}
  where $+\infty_{V}$ denotes the constant function $+\infty$ on
  $V$. Let the \emph{partially decorated surface $(\Delta,\lambda,u)$}
  be the surface obtained from the decorated surface
  $(\Delta,\lambda)$ by moving, for each vertex $v$, the horocycle at
  $v$ a distance $u(v)$ in the direction of the cusp $v$. If
  $u(v)=+\infty$, the horocycle at $v$ is missing. In particular, for
  $u\in\R^{V}$, the surface $(\Delta,\lambda,u)$ is just the decorated
  surface with Penner coordinates
  $(\Delta,\Lambda^{\Delta,\lambda}(u))$.
\end{definition}

If~\eqref{eq:local2} is the local Delaunay condition at edge
$e\in E_{\Delta}$ for the decorated surface $(\Delta,\lambda)$, then
the local Delaunay condition at $e$ for the partially decorated surface
$(\Delta,\lambda,u)$ is 
\begin{equation}
  \label{eq:localpar}
  \alpha e^{-u(v_{a})}+\alpha'e^{-u(v_{a'})}
  \leq(\beta+\beta')e^{-u(v_{b})}+(\gamma+\gamma')e^{-u(v_{c})},
\end{equation}
where $e^{-\infty}=0$. 

\begin{proposition}[flip algorithm for partially decorated surfaces]
  \label{prop:flippartial}
  An ideal Delaunay triangulation of the partially decorated surface
  $(\Delta,\lambda,u)$ can be found by Weeks's flip algorithm (see
  Theorem~\ref{thm:flip}) with the local condition~\eqref{eq:local2}
  replaced by~\eqref{eq:localpar}. An adjusted Delaunay triangulation
  can then be found by iteratively flipping all nonessential edges
  opposite an undecorated cusp until no such edge remains.
\end{proposition}

With Proposition~\ref{prop:Dellim} below, the correctness of this
algorithm for partially decorated surfaces can be deduced from the
correctness of the original algorithm for decorated surfaces.

\begin{proposition}
  \label{prop:Dellim}
  If $\Deltil$ is an adjusted Delaunay triangulation for the partially
  decorated surface $(\Delta,\lambda,u)$, then there is an $M\in\R$
  such that $\Deltil$ is also an ideal Delaunay triangulation for the
  decorated surfaces $(\Delta,\Lambda^{\Delta,\lambda}(\tilde{u}))$
  with $\tilde{u}\in\R^{V}$ satisfying
  \begin{alignat*}{2}
    \tilde{u}(v)&=u(v)\quad&
    &\text{if}\quad u(v)<+\infty,\\
    \tilde{u}(v)&>M\quad&
    &\text{if}\quad u(v)=+\infty.
  \end{alignat*}
\end{proposition}

This is a corollary of the characterization of punctured Delaunay
faces in terms of horocycle distances
(see Proposition~\ref{prop:punctured}).

\begin{definition}[generalized Penner coordinates]
  \label{def:genpenner}
  If $\Delta$ is an arbitrary triangulation of a partially decorated
  surface, then the surface is in general not determined by $\Delta$
  and the function $\lambda\in\Rbar^{E}$ of horocycle distances, which
  takes the value $+\infty$ if the horocycle at one or both ends is
  missing. However, if $\Delta$ is an adjusted Delaunay triangulation,
  then the pair $(\Delta,\lambda)$ determines the partially decorated surface
  uniquely (see Proposition~\ref{prop:punctured}), and we call such a
  pair \emph{generalized Penner coordinates}.
\end{definition}

\begin{definition}
  \label{def:Del2}
  Let $\Del$ also denote a function
  \begin{equation*}
    \Del:(\Delta,\lambda,u)\longmapsto (\Deltil,\lamtil)
  \end{equation*}
  that maps the parameters $(\Delta,\lambda,u)$ of a partially
  decorated surface to generalized Penner coordinates
  $(\Deltil,\lamtil)$ of the same partially decorated surface, where
  $\Deltil$ is an adjusted Delaunay triangulation. Hence
  \begin{equation*}
    \lamtil=\Lambda^{\Deltil,\tau_{\Delta,\Deltil}(\lambda)}(u)\,
    \in\Rbar^{E_{\Deltil}},
  \end{equation*}
  where
  $\tau_{\Delta,\Deltil}:\R^{E_{\Delta}}\rightarrow\R^{E_{\Deltil}}$
  is the chart transition function mapping Penner coordinates with
  respect to $\Delta$ to Penner coordinates with respect to $\Deltil$.
\end{definition}

Such a function $\Del$ can be computed by the modified flip algorithm
of Proposition~\ref{prop:flippartial}.

The bounds constraints in the variational principle of
Theorem~\ref{thm:variational} involve signed distances of horocycles in a
decorated surface, which are defined as follows:

\begin{definition}[signed distance of horocycles]
  \label{def:delta}
  Let $\delta_{\Delta,\lambda}(v_{1},v_{2})$ denote the signed
  distance of the horocycles at the vertices $v_{1},v_{2}\in V$ in the
  decorated surface with Penner coordinates $(\Delta,\lambda)$. More
  precisely, let
  \begin{equation}
    \label{eq:delta}
    \delta_{\Delta,\lambda}(v_{1},v_{2})=\min_{h_{1},h_{2}}d(h_{1},h_{2})
  \end{equation}
  where $d$ denotes the signed distance of horocycles in $H^{2}$ and
  the minimum is taken over all pairs of horocycles $h_{1}\not=h_{2}$
  in $H^{2}$ that are lifts of the horocycles at $v_{1}$ and~$v_{2}$,
  respectively.
\end{definition}

\begin{remark}
  The distance $\delta_{\Delta,\lambda}(v_{1},v_{2})$ is well defined
  and non-trivial even for $v_{1}=v_{2}$, but we will not need this.
\end{remark}

Proposition~\ref{prop:punctured} implies that the horocycle distances
$\delta_{\Delta,\lambda}$ can be calculated using the flip algorithm:

\begin{proposition}[Calculating $\delta_{\Delta,\lambda}$]
  \label{prop:calcdelta}
  Let $(\Delta,\lambda)$ be Penner coordinates for a decorated surface,
  let~$v_{1},v_{2}\in V$, let 
    \begin{equation*}
    u\in\Rbar^{V},\qquad
    u(v)=
    \begin{cases}
      0 & \text{if }v=v_{2}\\
      +\infty & \text{otherwise},
    \end{cases}
  \end{equation*}
  and let $(\Deltil,\lamtil,u)$ be the output of the modified flip
  algorithm (see Proposition~\ref{prop:flippartial}). Then $v_{1}$ is
  the central vertex of a punctured face of $\Deltil$ and all
  peripheral vertices are $v_{2}$. For any edge $e\in E_{\Deltil}$
  connecting $v_{1}$ and $v_{2}$
  \begin{equation*}
    \delta_{\Delta,\lambda}(v_{1},v_{2})=\lamtil_{e}.
  \end{equation*}
\end{proposition}

\section{Realizable coordinates}
\label{sec:realizable}

In this section, we characterize the Penner coordinates that may be
obtained from an ideal polyhedron by the construction described in
Section~\ref{sec:overview}: They are realizable coordinates by
Definition~\ref{def:realizable} (see
Proposition~\ref{prop:poly2realizable}). Conversely, for given
realizable coordinates, a corresponding ideal polyhedron is uniquely
determined by an explicit construction (see
Proposition~\ref{prop:realizable2poly}).

\begin{definition}[realizable coordinates]
  \label{def:realizable}
  Let $S$ be a complete finite area hyperbolic surface of
  genus $0$ with $n\geq 3$ cusps. \emph{Realizable coordinates for $S$
    with distinguished vertex $v_{\infty}\in V$} are a pair
  $(\Deltil,\lamtil)$ consisting of a triangulation~$\Deltil$ of
  $(S_{0},V)$ and a function $\lamtil=\Rbar^{E_{\Deltil}}$, where
  $\Rbar=\R\cup\{+\infty\}$, satisfying the following conditions (r1)
  and (r2):
  \begin{compactenum}[(r1)]
  \item $\Deltil$ is an adjusted Delaunay triangulation for a
    decoration of~$S$ with exactly one missing horocycle at
    $v_{\infty}$, and $(\Deltil,\lamtil)$ are the corresponding
    generalized Penner coordinates (see Definition~\ref{def:genpenner}).

  \item Let $\Deltilo$ be the subcomplex of $\Deltil$ consisting of all
    closed cells not incident with~$v_{\infty}$, i.e.,
    \begin{align}
      %\label{eq:Vo}
      V_{\Deltilo}&=V\setminus\{v_{\infty}\},\\
      \label{eq:Eo}
      E_{\Deltilo} &=\{e\in E_{\Deltil}\;|\; e \text{ not incident
                      with }v_{\infty}\},\\
      \label{eq:To}
      T_{\Deltilo} &=\{t\in T_{\Deltil}\;|\; t\text{ not incident
                      with }v_{\infty}\}.
    \end{align}
    Then either (r2a) or (r2b) are true:
    \begin{compactitem}[(r2a)]
    \item[(r2a)] $T_{\Deltilo}=\emptyset$, and $\Deltilo$ is a linear
      graph, that is, a graph of the form
      \begin{equation*}
        \bullet-\bullet-\cdots-\bullet.
      \end{equation*}
    \item[(r2b)] $T_{\Deltilo}\not=\emptyset$ and $\Deltilo$ is a
      triangulation of a closed disk. Moreover,
      \begin{alignat}{2}
        \label{eq:Thetatileq}
          \Thetatil_{v}&=2\pi\quad
          &&\text{if $v$ is an interior vertex of $\Deltilo$,}\\
        \label{eq:Thetatilineq}
          \Thetatil_{v}&\leq \pi\quad
          && \text{if $v$ is a boundary vertex of $\Deltilo$,}
      \end{alignat}
      where 
      \begin{equation}
        \label{eq:Thetatil}
        \Thetatil_{v}=\text{ sum of angles around $v$ in $\Deltilo{}$},
      \end{equation}
      measured in the piecewise euclidean metric that
      turns each triangle in $T_{\Deltilo}$ into a euclidean triangle
      and each edge in $e\in E_{\Deltilo}$ into a euclidean line segment of
      length $\elltil_{e}$, where
      \begin{equation}
        \label{eq:elltil}
        \elltil=e^{\frac{1}{2}\lamtil}.
      \end{equation}
    \end{compactitem}
  \end{compactenum}    
\end{definition}

\begin{proposition}[ideal polyhedron $\rightarrow$ realizable coordinates]
  \label{prop:poly2realizable}
  Let $P$ be a three-dimensional convex ideal polyhedron or a
  two-sided ideal polygon in $H^{3}$ realizing a surface
  $S\in\Teich_{0,n}$, which has Penner coordinates $(\Delta,\lambda)$
  for some decoration.  Let~$v_{\infty}$ be an ideal vertex of
  $P$. Let $\Deltil$ be a triangulation obtained by adding diagonals
  to triangulate all non-triangular faces of $P$. The choice of
  diagonals is arbitrary except for non-triangular faces incident with
  $v_{\infty}$, in which the diagonals incident with $v_{\infty}$ are
  be chosen. Let $s_{\infty}$ be a horosphere centered at
  $v_{\infty}$. For any other vertex $v$, let $s_{v}$ be the
  horosphere centered at $v$ and touching $s_{\infty}$. For each edge
  $e\in E_{\Delta}$ not incident with $v_{\infty}$, let $\lamtil_{e}$
  be the signed distance of the horospheres at the ends of $e$. If $e$
  is incident with $v_{\infty}$, let $\lamtil_{e}=\infty$.

  Then $(\Deltil,\lamtil)$ are realizable coordinates for $S$.
\end{proposition}

\begin{proof}
  In the case of a two-sided polygon, the statement follows easily
  from the characterization of punctured Delaunay faces,
  Proposition~\ref{prop:punctured}. It remains to consider the case
  of $P$ being a three-dimensional polyhedron.

  To show (r1), first note that for an edge $e$ not incident with
  $v_{\infty}$, $\lambda_{e}$ is also the intrinsic signed distance of
  the horocycles $h_{v}=s_{v}\cap P$ at the vertices of~$e$. It
  remains to show that $\Deltil$ is an adjusted Delaunay
  triangulation. First, the union of triangles incident with
  $v_{\infty}$ is the Delaunay face around the vertex $v_{\infty}$
  with vanished horocycle. Indeed, the horocycle $h_{\infty}$ touches
  the horocycles at adjacent vertices and does not meet the horocycles
  at all other vertices.
    
  It remains to show that the local Delaunay conditions (see
  Theorem~\ref{thm:localdel}) are satisfied for all edges between two
  triangles that are not incident with $v_{\infty}$. We may assume
  without loss of generality that the horosphere $s_{\infty}$ was
  chosen large enough so that the horospheres at the other vertices
  are pairwise disjoint. Consider a face $f$ of $P$ that is not
  incident with $v_{\infty}$. The point in the hyperbolic plane of $f$
  that is closest to $s_{\infty}$ is the center of the circle in this
  plane that touches all horospheres at the vertices of~$f$
  externally. (See Figure~\ref{fig:poly}: The points closest to
  $s_{\infty}$ are the highest points in the hemispheres.) Using this
  fact, it is not difficult to see that the local convexity of the
  edge $e$ in $P$ is equivalent to the local Delaunay condition (lD2)
  of Theorem~\ref{thm:localdel}.
  
  (r2) Since we assume that $P$ is a three-dimensional polyhedron,
  $\Deltilo$ is a triangulation of a closed disk. Now (r2b) follows by
  decomposing $P$ into ideal tetrahedra as shown in
  Figure~\ref{fig:idealtet}, one tetrahedron for each triangle in
  $\Deltilo$. The horosphere $s_{\infty}$ intersects these tetrahedra
  in euclidean triangles with side lengths~$\elltil$ determined
  by~\eqref{eq:elltil}. This implies~\eqref{eq:Thetatileq}
  and~\eqref{eq:Thetatilineq} because $P$ intersects $s_{\infty}$ in a
  convex euclidean polygon.
\end{proof}

\begin{proposition}[realizable coordinates $\rightarrow$ ideal polyhedron]
  \label{prop:realizable2poly}
  Let $(\Deltil,\lamtil)$ be realizable coordinates of a surface
  $S\in\Teich_{0,n}$. Let $\Deltilo$ be the subcomplex of $\Deltil$
  defined in Definition~\ref{def:realizable}~(r2).
  
  (i) If $\Deltilo$ is a triangulation of a closed disk, one obtains a
  polyhedral realization of $S$ as follows: Construct a decorated
  ideal tetrahedron as shown in Figure~\ref{fig:idealtet} for each
  triangle of $\Deltilo$. These ideal tetrahedra fit
  together to form a polyhedron that realizes $S$.

  (ii) If $\Deltilo$ is a linear graph then $\Deltil$ is a
  decomposition of $S$ into partially decorated ideal triangles that
  fit together to form a realization of~$S$ as two-sided ideal
  polygon.
\end{proposition}

We omit a detailed proof. The tetrahedra in (i) exist by
Lemma~\ref{lem:triang}. They fit together to form an ideal polyhedron
realizing $S$ by (r2b). That the polyhedron is convex follows from
inequality~\eqref{eq:Thetatilineq} and from the fact that~$\Deltil$ is
a Delaunay triangulation.

The realizable coordinates $(\Deltil,\lamtil)$ with distinguished
vertex $v_{\infty}$ that are obtained from a polyhedron or two-sided
polygon by Proposition~\ref{prop:poly2realizable} are not uniquely
determined:
\begin{compactitem}
\item Non-triangular faces that are not incident with
  $v_{\infty}$ may be triangulated in different ways. 
\item A different choice of horosphere $s_{\infty}$ leads to
  realizable coordinates
  $(\Deltil,\lamtil+h\mathbf{1}_{E_{\Deltil}})$ for some $h\in\R$.
\end{compactitem}
But these are the only sources of ambiguity: If $(\Deltil,\lamtil)$
and $(\Deltil',\lamtil')$ are both realizable coordinates obtained
from the same polyhedron or two-sided polygon with the same
distinguished vertex $v_{\infty}$ by
Proposition~\ref{prop:poly2realizable}, then
\begin{compactitem}[(u2)]
\item[(u1)] $\Deltil$ and $\Deltil'$ are both adjusted Delaunay
  triangulations of the same ideal Delaunay decomposition,
\item[(u2)]
  $\lamtil'=\tau_{\Deltil,\Deltil'}(\lamtil+h\mathbf{1}_{E_{\Deltil}})$
  for some $h\in\R$.
\end{compactitem}
Conversely, the polyhedra obtained from different realizable
coordinates $(\Deltil,\lamtil)$ and $(\Deltil',\lamtil')$ with the
same distinguished vertex $v_{\infty}$ are congruent (as polyhedra
marked by $(S_{0},V)$) if and only if conditions (u1) and (u2) are
satisfied.

\begin{definition}[equivalent realizable coordinates]
  \label{def:realizequiv}
  Realizable coordinates $(\Deltil,\lamtil)$ and $(\Deltil',\lamtil')$
  with distinguished vertex $v_{\infty}$ are \emph{equivalent} if they
  satisfy conditions (u1) and (u2).
\end{definition}

Realizable coordinates are equivalent if and only if they
correspond to congruent realizations.

\section{The variational principle}
\label{sec:variational}

In this section we present a variational principle
(see Theorem~\ref{thm:variational}) for Problem~\ref{prob:realize2} of
finding realizable coordinates. The variational principle involves the
function~$\Ecalbar^{v_{\infty}}_{\Delta,\lambda}(u)$
(see Definition~\ref{def:Ecalbar}). The variables
$u\in\R^{V\setminus\{v_{\infty}\}}$ parametrize a part of the extended
fiber of $\widetilde{\Teich}_{g,n}$ over the surface
$(\Delta,\lambda)$. More precisely, they parametrize the horocycle
decorations of that surface with missing horocycle at
$v_{\infty}$. The variables are subject to bounds constraints, which
ensure that the horocycles do not intersect an arbitrary but fixed
horocycle at $v_{\infty}$. The definition of
$\Ecalbar^{v_{\infty}}_{\Delta,\lambda}(u)$ requires some
preparation. After the following brief summary, a detailed account
begins with Definition~\ref{def:f}.

The function $f(x_{1},x_{2},x_{3})$ (see Definition~\ref{def:f}) provides
a variational encoding of euclidean trigonometry: If the variables
$x_{i}$ are the logarithmic side lengths of a euclidean triangle, then
the partial derivatives of $f$ are its angles
(see Proposition~\ref{prop:f}). 

Using this building block, the function
$\lambda\mapsto\Hsf_{\Theta}(\Delta,\lambda)$
(see Definition~\ref{def:Hsf}) is defined on an open subset $A_{\Delta}$
of the decorated {\Tm} space $\decTeich_{g,n}$. The function
$\Esf_{\Theta,\Delta,\lambda}(u)$ (see Definition~\ref{def:Esf}) is the
restriction of $\Hsf_{\Theta}(\Delta,\,\cdot\,)$ to the intersection
of $A_{\Delta}$ with the fiber of $\decTeich_{g,n}$ over the surface
$(\Delta,\lambda)$. This function was already used in a previous
article (see Remark~\ref{rem:Esf}).

Next we extend the function $\Hsf_{\Theta}(\Delta,\lambda)$ to the
whole decorated \Tm{} space and the function
$\Esf_{\Theta,\Delta,\lambda}(u)$ to a whole fiber by adapting the
triangulation appropriately. To this end, consider the restriction of
$\Hsf_{\Theta}(\Delta,\,\cdot\,)$ to the closed Penner cell of all
$\lambda\in\R^{E_{\Delta}}$ for which $\Delta$ is a Delaunay
triangulation of the decorated surface~$(\Delta,\lambda)$. For
different triangulations $\Delta$, these restrictions of
$\Hsf_{\Theta}(\Delta,\,\cdot\,)$ fit together to define a
$C^{2}$ function $\Hcal_{\Theta}$ on the whole decorated {\Tm}
space~$\decTeich_{g,n}$ (see Corollary~\ref{cor:Hcal}). We denote by
$\Hcal_{\Theta,\Delta}(\lambda)$ the representation of
$\Hcal_{\Theta}$ in the global Penner coordinate system belonging to
the ideal triangulation $\Delta$ (see Definition and
Proposition~\ref{defprop:Hcal}). The convex
function~$\Ecal_{\Theta,\Delta,\lambda}(u)$
(see Definition~\ref{def:Ecal}) is the restriction of
$\Hcal_{\Theta,\Delta}(\lambda)$ to the fiber of $\decTeich_{g,n}$
over $(\Delta,\lambda)$. Finally,
$\Ecalbar^{v_{\infty}}_{\Delta,\lambda}(u)$ is obtained by setting
$\Theta_{v_{\infty}}=0$ and $\Theta_{v}=2\pi$ for $v\not=v_{\infty}$,
and taking the limit $u_{v_{\infty}}\rightarrow+\infty$
(see Definition~\ref{def:Ecalbar}).

\begin{figure}
  \labellist
  \small\hair 2pt
  \pinlabel {$x_{1}$} [l] at 241 120
  \pinlabel {$x_{2}$} [b] at 120 240
  \pinlabel {$\Acal$} [ ] at 137 142
  \endlabellist
  \centering
  \includegraphics[width=0.4\linewidth]{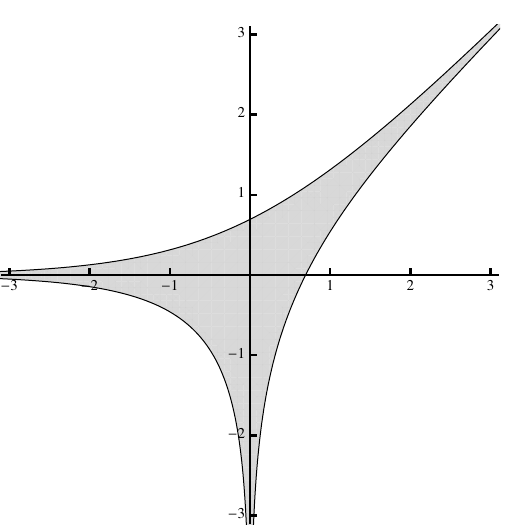}
  \caption{Intersection of the domain $\Acal$ with the plane
    $x_{3}=0$. The domain~$\Acal$ is invariant with respect to translations in
    the scaling direction $(1,1,1)$.}
  \label{fig:Acal}
\end{figure}

\begin{definition}[the triangle function $f$]
  \label{def:f}
  Let $f$ be the
  function
  \begin{gather}
    \notag
    f:\R^{3}\supseteq\Acal\longrightarrow\R,\\
    f(x_{1},x_{2},x_{3})=\alpha_{1} x_{1} + \alpha_{2} x_{2} +
    \alpha_{3} x_{3} + \lob(\alpha_{1}) + \lob(\alpha_{2}) +
    \lob(\alpha_{3}),
    \label{eq:f}
  \end{gather}
  where 
  \begin{equation*}
    \Acal=
    \left\{
      \begin{pmatrix}
        x_{1}\\x_{2}\\x_{3}
      \end{pmatrix}
      \in\R^{3}
      \ \middle|\ 
    \begin{aligned}
      e^{x_{1}}&>e^{x_{2}}+e^{x_{3}},\\
      e^{x_{2}}&>e^{x_{3}}+e^{x_{1}},\\
      e^{x_{3}}&>e^{x_{1}}+e^{x_{2}}
    \end{aligned}
    \right\}
  \end{equation*}
  (see Figure~\ref{fig:Acal}), $\lob$ is Milnor's Lobachevsky
  function~\cite{milnor82},
  \begin{equation}
    \label{eq:ML}
    \lob(\alpha) = -\int_{0}^{\alpha}\log\big|2\sin(t)\big|\,dt
  \end{equation}
  (see Figure~\ref{fig:lobachevskyplot}), and $\alpha_{j}$ are the
  angles in a triangle with sides~$e^{x_{j}}$ as shown in
  Figure~\ref{fig:triangle_f}.
\end{definition}

\begin{figure}
  \begin{minipage}[c]{0.6\linewidth}
    \labellist
    \small\hair 2pt
    \pinlabel {$\alpha$} [ ] at 280 81.5
    \pinlabel {$y$} [ ] at 28 168
    \pinlabel {$y=\lob(\alpha)$} [ ] at 130 150
    \endlabellist
    \centering
    \includegraphics[width=0.8\linewidth]{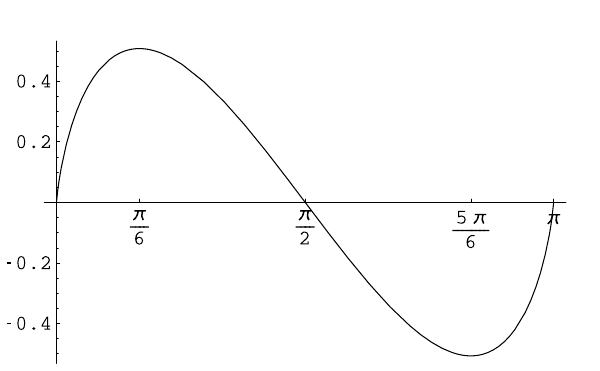}
  \end{minipage}
  \hfill
  \begin{minipage}[c]{0.35\linewidth}
    \labellist
    \small\hair 2pt
    \pinlabel {$e^{x_{1}}$} [bl] at 61 35
    \pinlabel {$e^{x_{2}}$} [br] <1pt,-1pt> at 23 31
    \pinlabel {$e^{x_{3}}$} [t] at 44 5
    \pinlabel {$\alpha_{1}$} [ ] at 13 7
    \pinlabel {$\alpha_{2}$} [ ] at 71 14.5
    \pinlabel {$\alpha_{3}$} [ ] at 43 45
    \endlabellist
    \centering
    \includegraphics{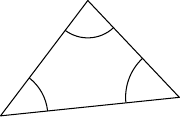}
  \end{minipage}\\
  \begin{minipage}[t]{0.6\linewidth}
    \caption{Milnor's Lobachevsky function $\lob$ is $\pi$-periodic,
      odd, and analytic except at $\alpha\in\pi\,\Z$, where the
      derivative tends to $+\infty$.}
    \label{fig:lobachevskyplot}
  \end{minipage}
  \hfill
  \begin{minipage}[t]{0.35\linewidth}
    \caption{Triangle with sides $e^{x_{1}}$, $e^{x_{2}}$, $e^{x_{3}}$}
    \label{fig:triangle_f}    
  \end{minipage}
\end{figure}

\begin{proposition}[properties of $f$]
  \label{prop:f}
  (i) The function $f$ is analytic, and it satisfies the \emph{scaling
    relation}
  \begin{equation}
    \label{eq:fscale}
    f(x_{1}+h,x_{2}+h,x_{3}+h) = f(x_{1},x_{2},x_{3}) + \pi h.
  \end{equation}
  
  (ii) The partial derivatives of $f$ are
  \begin{equation}
    \label{eq:df}
    \frac{\partial f}{\partial x_{i}}=\alpha_{i}.
  \end{equation}

  (iii) The second derivative of $f$ is 
  \begin{equation}
    \label{eq:d2f}
    D^{2}f|_{x}=\cot\alpha_{1}(dx_{2}-dx_{3})^{2}+\cot\alpha_{2}(dx_{3}-dx_{1})^{2}
    +\cot\alpha_{3}(dx_{1}-dx_{2})^{2}
  \end{equation}
  
  (iv) The second derivative $D^{2}f|_{x}$ is positive semidefinite
  with one-dimensional kernel spanned by $(1,1,1)$. In particular, $f$
  is locally convex.
\end{proposition}

\noindent%
See~\cite[Sec.~4.2]{bobenko15} for proofs.

\begin{definition}[the function $\Hsf_{\Theta}$]
  \label{def:Hsf}
  For a triangulation $\Delta$ of $(S_{g},V)$ and $\Theta\in\R^{V}$,
  let 
  \begin{gather}
    \notag
    \Hsf_{\Theta}(\Delta,\,\cdot\,):A_{\Delta}\longrightarrow\R,
    \\
    \begin{split}
      \Hsf_\Theta{}(\Delta,\lambda)= 
      \sum_{t\in T_{\Delta}} 
      2f \Big(
      \frac{\lambda_{e_{1}(t)}}{2}, 
      \frac{\lambda_{e_{2}(t)}}{2},
      \frac{\lambda_{e_{3}(t)}}{2} 
      \Big) 
      &-\pi\sum_{e\in E_{\Delta}}\lambda_{e} \\
      &-\sum_{v\in V}\Theta_{v}\log c_{v}(\Delta,\lambda),
    \end{split}
    \label{eq:Hsf}
  \end{gather}
  where $A_{\Delta}\subseteq\R^{E_{\Delta}}$ is the subset
  \begin{equation}
    \label{eq:ADelta}
    A_{\Delta}=\big\{\lambda\in\R^{E_{\Delta}}
    \;\big|\;
    \tfrac{1}{2}\big(
    \lambda_{e_{1}(t)},
    \lambda_{e_{2}(t)},
    \lambda_{e_{3}(t)}
    \big)\in \Acal
    \text{ for all }
    t\in T_{\Delta}
    \big\},
  \end{equation}
  and $c_{v}$ is defined by~\eqref{eq:c}.
\end{definition}

\begin{proposition}[properties of $\Hsf_{\Theta}$]
  \label{prop:Hsf}
  (i) The function $\Hsf_{\Theta}(\Delta,\,\cdot\,)$ is analytic and
  satisfies the scaling relation
  \begin{equation}
    \label{eq:Hsfscale}
    \Hsf_{\Theta}(\Delta,\lambda+h\,\mathbf{1}_{E_{\Delta}})
    =\Hsf_{\Theta}(\Delta,\lambda)+
    h\,\pi
    \Big(
    |T_{\Delta}|-|E_{\Delta}|+\frac{1}{2\pi}\sum_{v\in V}\Theta_{v}
    \Big).
  \end{equation}
  
  (ii) For $\Theta=0$, the function $\Hsf_{0}(\Delta,\,\cdot\,)$ is
  convex, and the kernel of the positive semi-definite second
  derivative is one-dimensional and spanned by
  $\mathbf{1}_{E_{\Delta}}$.
\end{proposition}

This follows immediately from the corresponding properties of $f$ and
the scaling behavior of $c_{v}$:
\begin{equation}
  \label{eq:cscale}
  c_{v}(\Delta,\lambda+h\,\mathbf{1}_{E_{\Delta}})=
  c_{v}(\Delta,\lambda)-\tfrac{1}{2}h.
\end{equation}

\begin{definition}[the function $\Esf_{\Theta,\Delta,\lambda}$]
  \label{def:Esf}
  Let $\Esf_{\Theta,\Delta,\lambda}$ be the restriction of
  $\Hsf_{\Theta}$ to the fiber of~$\decTeich_{g,n}$
  over~$(\Delta,\lambda)$ parametrized by $\Lambda^{\Delta,\lambda}$
  as defined by~\eqref{eq:Lambda}, that is,
  \begin{gather}
    \notag
    \Esf_{\Theta,\Delta,\lambda}:
    \{u\in\R^{V}\,|\, \Lambda^{\Delta,\lambda}(u)\in A_{\Delta}\}
    \longrightarrow\R\\
      \Esf_{\Theta,\Delta,\lambda}(u)=\Hsf_{\Theta}(\Delta,\Lambda^{\Delta,\lambda}(u)).
    \label{eq:Esf}
  \end{gather}
\end{definition}

\begin{remark}
  \label{rem:Esf}
  The function $\Esf_{\Theta,\Delta,\lambda}(u)$ is up to an additive
  constant equal to the function $E_{\mathsf{T},\Theta,\lambda}(u)$
  (with $\mathsf{T}=\Delta$) defined in the previous article
  \cite[eq.~(4-6)]{bobenko15}. In that article, the domain of
  $E_{\mathsf{T},\Theta,\lambda}(u)$ is extended to the whole $\R^{V}$
  by exploiting the fact that the function $f$ can be extended to a
  convex function on the whole $\R^{3}$. Here, we do not need this
  extension. Instead, we will extend the functions
  $\Hsf_{\Theta}(\Delta,\,\cdot\,)$ and hence also
  $\Esf_{\Theta,\Delta,\lambda}(u)$ by changing the triangulation
  (see Definition and Proposition~\ref{defprop:Hcal} and
  Definition~\ref{def:Ecal}).
\end{remark}

\begin{proposition}[properties of $\Esf_{\Theta,\Delta,\lambda}$]
  \label{prop:Esf}
  (i) The function $\Esf_{\Theta,\Delta,\lambda}$ is analytic and satisfies
  the scaling relation
  \begin{equation}
    \label{Esfscale}
    \begin{aligned}
      \Esf_{\Theta,\Delta,\lambda}(u+h\,\mathbf{1}_{V})
      =\Esf_{\Theta,\Delta,\lambda}(u)
      +h\,2\pi
      \Big(
      |T_{\Delta}|-|E_{\Delta}|+\frac{1}{2\pi}\sum_{v\in V}\Theta_{v}
      \Big).
    \end{aligned}
  \end{equation}
  
  (ii) For a vertex $v\in V$, the partial derivative of
  $\Esf_{\Delta,\lambda}(u)$ with respect to $u_{v}$ is
  \begin{equation}
    \label{eq:Esfpartial}
    \frac{\partial}{\partial u_{v}}\,\Esf_{\Theta,\Delta,\lambda}(u)
    =\Theta_{v}-\widetilde{\Theta}_{v},
  \end{equation}
  where $\widetilde{\Theta}_{v}$ is the angle sum around vertex $v$
  measured in the piecewise euclidean metric that turns every triangle
  in $T_{\Delta}$ into a euclidean triangle and every edge
  $e\in E_{\Delta}$ into a straight line segment of length
  $\elltil_{e}$, where $\elltil$ is defined by~\eqref{eq:elltil} with
  \begin{equation*}
    \lamtil=\Lambda^{\Delta,\lambda}(u).
  \end{equation*}

  (iii) The second derivative of $\Esf_{\Theta,\Delta,\lambda}$ at $u$ is
  \begin{equation}
    \label{eq:d2Esf}
      \sum_{v,v'\in V}
      \frac{\partial^{2}\Esf_{\Theta,\Delta,\lambda}} {\partial
        u_{v}\partial u_{v'}} \,du_{v}\,du_{v'} =
      \frac{1}{4}\,\sum_{e\in E_{\Delta}}
      (\cot\alpha_{e}(u)+\cot\alpha_{e}'(u))
      \big(du_{v_{1}(e)}-du_{v_{2}(e)}\big)^{2},
  \end{equation}
  where $\alpha_{e}(u)$ and $\alpha'_{e}(u)$ are the angles opposite
  edge $e$ in the piecewise euclidean metric described in (ii).

  (iv) The function $\Esf_{\Theta,\Delta,\lambda}$ is locally convex. The
  second derivative is positive semidefinite with one-dimensional
  kernel spanned by $\mathbf{1}_{V}$. 
\end{proposition}

\begin{remark}
  \label{rem:d2Esf}
  The second derivative of $\Esf_{\Theta,\Delta,\lambda}$ is the
  Dirichlet energy of the piecewise linear function taking the value
  $u_{v}$ at vertex $v$~\cite[eq.~(8)]{duffin59} \cite{pinkall93}. A
  satisfactory explanation for this coincidence seems to be unknown.
\end{remark}

\begin{proof}[Proof of Proposition~\ref{prop:Esf}]
  (i) and (iv) follow from the corresponding properties of
  $\Hsf_{\Theta}$ (see Proposition~\ref{prop:Hsf}), the scaling behavior
  of $\Lambda^{\Delta,\lambda}$,
  \begin{equation}
    \label{eq:Lambdascale}
    \Lambda^{\Delta,\lambda}(u+h\,\mathbf{1}_{V})=
    \Lambda^{\Delta,\lambda}(u)+2h\cdot \mathbf{1}_{E_{\Delta}},
  \end{equation}
  and the equation
  \begin{equation}
    \label{eq:cu}
    \log c_{v}(\Delta,\Lambda^{\Delta,\lambda}(u))
    =\log c_{v}(\Delta,\lambda)-u_{v}.
  \end{equation}
  
  (ii) follows from \eqref{eq:df} and \eqref{eq:cu} by a direct
  calculation. Note that if $(i,j,k)$ is a permutation of $(1,2,3)$
  and
  \begin{equation*}
    x_{i}=\frac{\lambda_{i}+u_{j}+u_{k}}{2}
  \end{equation*}
  then, with the notation of Figure~\ref{fig:ideal_triang},
  \begin{equation}
    \label{eq:dfdu}
    \frac{\partial}{\partial u_{i}}\,
    2f(x_{1},x_{2},x_{3})
    =\alpha_{j}+\alpha_{k}=\pi-\alpha_{i}.
  \end{equation}

  (iii) follows from equations~\eqref{eq:d2f}, \eqref{eq:Hsf},
  and~\eqref{eq:cu} by a direct calculation (see
  also~\cite[Prop.~4.1.6]{bobenko15}).
\end{proof}

The functions $\Hsf_{\Theta}(\Delta,\,\cdot\,)$, restricted to the respective
Penner cells of $\Delta$, fit together to form a single
$C^{2}$ function on the decorated \Tm{} space:

\begin{defprop}[the function $\Hcal_{\Theta,\Delta}$]
  \label{defprop:Hcal}
  For a triangulation $\Delta$ of $(S_{g},V)$, let
  $\Hcal_{\Theta,\Delta}$ be the function
  \begin{gather}
    \notag
    \Hcal_{\Theta,\Delta}:\R^{E_{\Delta}}\longrightarrow\R,\\
    \Hcal_{\Theta,\Delta}(\lambda)=\Hsf_{\Theta}(\Del(\Delta,\lambda))
    \label{eq:Hcal}\,.
  \end{gather}
  ($\Del(\Delta,\lambda)$ is defined in Definition~\ref{def:Del}.)
  The function $\Hcal_{\Theta,\Delta}$ is well defined, twice
  continuously differentiable, analytic in each open Penner cell of
  $\decTeich_{g,n}$, and satisfies the scaling relations
  \begin{equation}
    \label{eq:Hcalscale}
    \Hcal_{\Theta,\Delta}(\lambda+h\,\mathbf{1}_{E_{\Delta}})
    =\Hcal_{\Theta,\Delta}(\lambda)
    +h\,\pi
    \Big(
    |T_{\Delta}|-|E_{\Delta}|+\frac{1}{2\pi}\sum_{v\in V}\Theta_{v}
    \Big).
  \end{equation}
\end{defprop}

\begin{corollary}
  \label{cor:Hcal}
  There is a $C^{2}$ function
  $\Hcal_{\Theta}:\decTeich_{g,n}\rightarrow\R$ on the decorated {\Tm}
  space, which is analytic on each open Penner cell, such that for
  each ideal triangulation $\Delta$, the function
  $\Hcal_{\Theta,\Delta}$ is the representation of $\Hcal_{\Theta}$ in
  the global Penner coordinate chart belonging to~$\Delta$. The
  function $\Hcal_{\Theta}$ is invariant under the action of the
  mapping class group.
\end{corollary}

\begin{proof}[Proof of Proposition~\ref{defprop:Hcal}]
  The right hand side of equation~\eqref{eq:Hcal} is well-defined for
  all $\lambda\in\R^{E_{\Delta}}$ because
  $(\Deltil,\lamtil)=\Del(\Delta,\lambda)$ implies
  $\lamtil\in A_{\Deltil}$ by Lemma~\ref{lem:triang}.

  The function $\Hcal_{\Theta,\Delta}$ is analytic on open
  Penner cells because the functions $\Hsf_{\Theta}(\Deltil,\,\cdot\,)$ are
  analytic for all triangulations $\Deltil$, and so are the chart
  transition functions $\tau_{\Delta,\Deltil}:\lambda\mapsto\lamtil$
  for Penner coordinates with respect to different triangulations
  $\Delta$ and $\Deltil$.

  If the Delaunay triangulation for $(\Delta,\lambda)$ is not unique,
  then the value of the right hand side of~\eqref{eq:Hcal} and its
  first two derivatives are independent of the choice of Delaunay
  triangulation. This follows from Lemma~\ref{lem:Hsf}, which we defer
  to Section~\ref{sec:difflemma}. It also implies that
  $\Hcal_{\Theta,\Delta}$ is twice continuously differentiable.

  The scaling relation~\eqref{eq:Hcalscale} follows from the scaling
  relation for $\Hsf_{\Theta}$ (see Proposition~\ref{prop:Hsf}).
\end{proof}

\begin{definition}
  \label{def:Ecal}
  Let $\Ecal_{\Theta,\Delta,\lambda}$ be the restriction of
  $\Hcal_{\Theta,\Delta}$ to the fiber of~$\decTeich_{g,n}$
  over~$(\Delta,\lambda)$ parametrized by $\Lambda^{\Delta,\lambda}$,
  that is,
  \begin{gather}
    \notag
    \Ecal_{\Theta,\Delta,\lambda}:\R^{V}\longrightarrow\R,\\
    \label{eq:Ecal}
    \Ecal_{\Theta,\Delta,\lambda}(u)
    =\Hcal_{\Theta,\Delta}(\Lambda^{\Delta,\lambda}(u)).
  \end{gather}
\end{definition}

\begin{proposition}[properties of $\Ecal_{\Theta,\Delta,\lambda}$]
  \label{prop:Ecal}
  (i) The function $\Ecal_{\Theta,\Delta,\lambda}$ is twice
  continuously differentiable, analytic in the interior of each Penner
  cell, and satisfies the scaling relations
  \begin{equation}
    \label{eq:Ecalscale}
    \Ecal_{\Theta,\Delta,\lambda}(u+h\,\mathbf{1}_{V})
    =\Ecal_{\Theta,\Delta,\lambda}(u)
    +h\,2\pi
    \Big(
    |T_{\Delta}|-|E_{\Delta}|+\frac{1}{2\pi}\sum_{v\in V}\Theta_{v}
    \Big).
  \end{equation}

  (ii) The partial derivatives are
  \begin{equation}
    \label{eq:Ecalpartial}
    \frac{\partial}{\partial u_{v}}\,
    \Ecal_{\Theta,\Delta,\lambda}(u)
    =\Theta_{v}-\widetilde{\Theta}_{v},
  \end{equation}
  where $\widetilde{\Theta}_{v}$ is the total angle at $v$ when
  $S_{g}$ is equipped with the piecewise euclidean metric that turns
  every triangle in $T_{\Deltil}$ into a euclidean triangle and every
  edge $e\in E_{\Deltil}$ into a straight line segment of length
  $\tilde{\ell}_{e}$, where $\elltil$ is defined by \eqref{eq:elltil}
  and
  \begin{equation}
    \label{eq:DeltillamtilEcal}
    (\Deltil,\lamtil)=
    \Del(\Delta,\Lambda^{\Delta,\lambda}(u)).
  \end{equation}

  (iii) The second derivative of $\Ecal_{\Theta,\Delta,\lambda}$ at
  $u$ satisfies
    \begin{equation}
    \label{eq:d2Ecal}
    D^{2}\Ecal_{\Theta,\Delta,\lambda}\big|_{u}
    =D^{2}\Esf_{\Theta,\Deltil,\lamtil}\big|_{u},
  \end{equation}
  with $(\Deltil,\lamtil)$ defined by~\eqref{eq:DeltillamtilEcal}.

  (iv) The function $\Ecal_{\Theta,\Delta,\lambda}$ is convex. The
  second derivative is positive semidefinite with one-dimensional
  kernel spanned by~$\mathbf{1}_{V}$.
\end{proposition}

\begin{remark}
  \label{rem:d2Ecal}
  By equation~\eqref{eq:d2Ecal}, one can calculate the second
  derivative $D^{2}\Ecal_{\Theta,\Delta,\lambda}|_{u}$ by applying the
  flip algorithm (see Theorem~\ref{thm:flip}) to determine
  $(\Deltil,\lamtil)$ and then equation~\eqref{eq:d2Esf} for the
  second derivative of $\Esf_{\Theta,\Deltil,\lamtil}$.
\end{remark}

\begin{proof}
  The claims follow from the corresponding properties of
  $\Hcal_{\Theta,\Delta}$ and $\Esf_{\Theta,\Delta,\lambda}$. Note
  that the scaling action of $u\in\R^{V}$ on $\R^{E_{\Delta}}$
  commutes with the chart transition functions
  $\tau_{\Delta,\Deltil}$:
  \begin{equation}
    \label{eq:taucommute}
    \tau_{\Delta,\Deltil}(\Lambda^{\Delta,\lambda}(u))
    =\Lambda^{\Deltil,\tau_{\Delta,\Deltil}(\lambda)}(u).
  \end{equation}
  So with
  \begin{equation*}
    (\Deltil,\lamtil)=\Del(\Delta,\Lambda^{\Delta,\lambda}(u))
    =(\Deltil,\tau_{\Delta,\Deltil}\Lambda^{\Delta,\lambda}(u))
    =(\Deltil,\Lambda^{\Deltil,\tau_{\Delta,\Deltil}(\lambda)}(u))
  \end{equation*}
  one obtains
  \begin{equation}
    \label{eq:EcalEsf}
    \begin{split}
      \Ecal_{\Theta,\Delta,\lambda}(u)
      &=\Hcal(\Lambda^{\Delta,\lambda}(u))
      =\Hsf_{\Theta}(\Del(\Delta,\Lambda^{\Delta,\lambda}(u)))
      =\Hsf_{\Theta}(\Deltil,\Lambda^{\Deltil,\tau_{\Delta,\Deltil}(\lambda)}(u))\\
      &=\Esf_{\Theta,\Deltil,\tau_{\Delta,\Deltil}(\lambda)}(u).
    \end{split}
  \end{equation}

  Statements (ii) and (iii) follow with~\eqref{eq:EcalEsf} from the
  corresponding properties of $\Esf_{\Theta,\Delta,\lambda}$ (see
  Proposition~\ref{prop:Esf}).
\end{proof}

To formulate the variational principle of
Theorem~\ref{thm:variational} and for the variational existence proofs
(see Sections~\ref{sec:proof} and~\ref{sec:highergenus}) we need to
consider limits of $\Ecal_{\Theta,\Delta,\lambda}(u)$ as some
variables $u_{v}$ tend to infinity. It is enough to consider
$\Ecal_{0,\Delta,\lambda}$, that is, the case $\Theta=0$, because by
equation~\eqref{eq:cu},
\begin{equation}
  \label{eq:EcalThetaEcal0}
  \Ecal_{\Theta,\Delta,\lambda}(u)=\Ecal_{0,\Delta,\lambda}(u)
  -\sum_{v\in V}
  \Theta_{v}\big(\log c_{v}(\Delta,\lambda)-u_{v}\big).
\end{equation}

\begin{lemma}[limits of $\Ecal_{0,\Delta,\lambda}$]
  \label{lem:limEcal0}
  Let $\ubar\in\Rbar^{V}$, with $\Rbar$ defined by~\eqref{eq:Rbar},
  and assume $\ubar_{v}<+\infty$ for at least one vertex $v\in V$. Then
  \begin{equation}
    \label{eq:limEcal0}
    \lim_{u\rightarrow\ubar}\Ecal_{0,\Delta,\lambda}(u)=
    \sum_{t\in T_{\Deltilo}}
    2f \Big( \frac{\lamtil_{e_{1}(t)}}{2},
    \frac{\lamtil_{e_{2}(t)}}{2}, \frac{\lamtil_{e_{3}(t)}}{2} \Big)
    -
    \pi\sum_{e\in E_{\Deltilo}}\lamtil_{e},
  \end{equation}
  where 
  \begin{equation*}
    (\Deltil,\lamtil)=\Del(\Delta,\lambda,u), 
  \end{equation*}
  and $\Deltilo$ is the subcomplex of the adjusted Delaunay
  triangulation $\Deltil$ consisting of all closed cells that are not
  incident with an undecorated vertex, that is,
  \begin{align*}
    V_{\Deltilo}&=\{v\in V\;|\;u_{v}<+\infty\},\\
    E_{\Deltilo}&=\{e\in E_{\Deltil}\;|\;
                  \text{vertices of $e$ are contained in }V_{\Deltilo} 
                  \},\\
    T_{\Deltilo}&=\{t\in T_{\Deltil}\;|\;
                  \text{vertices of $t$ are contained in }V_{\Deltilo} 
                  \}.
  \end{align*}
\end{lemma}

\begin{corollary}
  \label{cor:limEcal}
  If $\Theta\geq 0$ and $\Theta_{v}>0$ for at least one $v\in V$
  with $\ubar_{v}=+\infty$, then 
  \begin{equation*}
    \lim_{u\rightarrow\ubar}\Ecal_{\Theta,\Delta,\lambda}(u)=+\infty.
  \end{equation*}
\end{corollary}

\begin{proof}[Proof of Lemma~\ref{lem:limEcal0}]
  By Akiyoshi's Theorem~\ref{thm:akiyoshi}, only finitely many ideal
  Delaunay decompositions arise from different decorations of the
  surface $(\Delta,\lambda)$. It is therefore enough to consider the
  limit of $\Ecal_{0,\Delta,\lambda}(u)$ as $u$ tends to $\ubar$ in
  the subset
  \begin{equation}
    \label{eq:usubset}
    \big\{u\in\R^{V}\;\big|\;
    \Delbar\text{ is an ideal Delaunay triangulation for }
    \big(\Delta,\Lambda^{\Delta,\lambda}(u)\big)
    \big\}\subseteq\R^{V}
   \end{equation}
   for some fixed ideal triangulation $\Delbar$. Then $\Delbar$ is
   also a Delaunay triangulation of the partially decorated surface
   $(\Delta,\lambda,\ubar)$, because the local Delaunay
   conditions~\eqref{eq:localpar} are non-strict inequalities, both
   sides of which extend continuously to $\Rbar^{V}$. In particular,
   $\Delbar$ and the adjusted Delaunay triangulation $\Deltil$ are
   ideal Delaunay triangulations of the same ideal Delaunay
   decomposition. For $u$ in the subset~\eqref{eq:usubset},
  \begin{multline}
    \label{eq:Ecal0H0}
    \Ecal_{0,\Delta,\lambda}(u)=\Hsf_{0}(\Delbar,\lambar(u))\\
    =\sum_{t\in T_{\Delbar}} 
      2f \Big(
      \frac{\lambar_{e_{1}(t)}(u)}{2}, 
      \frac{\lambar_{e_{2}(t)}(u)}{2},
      \frac{\lambar_{e_{3}(t)}(u)}{2} 
      \Big) 
      -\pi\sum_{e\in E_{\Delbar}}\lambar_{e}(u)
  \end{multline}
  where 
  \begin{equation*}
    \lambar(u)=\tau_{\Delta,\Delbar}\circ\Lambda^{\Delta,\lambda}(u).
  \end{equation*}
  In particular, for each triangle $t\in T_{\Delbar}$ and all $u$ in the
  subset~\eqref{eq:usubset}, 
  \begin{equation}
    \label{eq:lambartu}
    \tfrac{1}{2}
    \big(
    \lambar_{e_{1}(t)}(u), 
    \lambar_{e_{2}(t)}(u),
    \lambar_{e_{3}(t)}(u) 
    \big)
  \end{equation}
  is contained in $\Acal$ (see Definition~\ref{def:f}). There are
  three possibilities:
  \begin{compactenum}[(i)]
  \item Triangle $t$ is not contained in a punctured Delaunay cell of
    $(\Delta,\lambda,\ubar)$. Then~\eqref{eq:lambartu}
    converges to a point in $\Acal$.
  \item Triangle $t$ is contained in a punctured Delaunay cell of
    $(\Delta,\lambda,\ubar)$ and $t$ is incident with the undecorated
    central vertex. Then~\eqref{eq:lambartu} goes to infinity in
    $\Acal$. If $e_{i}(t)$ is the edge opposite the central vertex,
    and $e_{j}(t)$, $e_{k}(t)$ are the edges incident with the central
    vertex, then $\lambar_{e_{i}(t)}(u)$ has a finite limit while
    $\lambar_{e_{j}(t)}(u)$ and $\lambar_{e_{k}(t)}(u)$ tend to
    $+\infty$.
  \item Triangle $t$ is contained in a punctured Delaunay cell of
    $(\Delta,\lambda,\ubar)$ and $t$ is not incident with the
    undecorated central vertex. Then~\eqref{eq:lambartu} converges to
    a boundary point $(\xbar_{1},\xbar_{2},\xbar_{3})\in\partial\!\Acal$, that is,
    for some permutation $(i,j,k)$ of $(1,2,3)$,
    \begin{equation*}
      e^{\xbar_{i}}=
      e^{\xbar_{j}}+
      e^{\xbar_{k}}
    \end{equation*}
    (see Figures~\ref{fig:lambdaell} and~\ref{fig:triangdegen}).
  \end{compactenum}
  \begin{figure}
    \centering
    \includegraphics{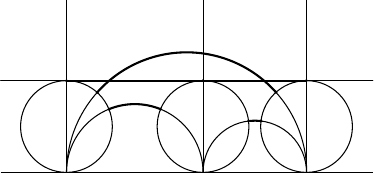}
    \caption{Ideal triangle contained in punctured Delaunay face, but
      not incident with the central vertex}
    \label{fig:triangdegen}
  \end{figure}

  Now equation~\eqref{eq:limEcal0} for the limit follows
  from equation~\eqref{eq:Ecal0H0} and the following limits
  of the function~$f$.

  \begin{compactenum}[(a)]
  \item As
    $(x_{1},x_{2},x_{3})\longrightarrow(\xbar_{1},+\infty,+\infty)$ in
    $\Acal$,
    \begin{equation}
      \label{eq:limf}
      f(x_{1},x_{2},x_{3})-\frac{\pi}{2}(x_{2}+x_{3})\longrightarrow 0.
    \end{equation}
    To see this, note that
    \begin{equation*}
      (\alpha_{1},\alpha_{2},\alpha_{3})\longrightarrow 
      \Big(0,\tfrac{\pi}{2}+\delta,\tfrac{\pi}{2}-\delta\Big)
    \end{equation*}
    for some $\delta\in[0,\frac{\pi}{2}]$. So
    \begin{equation*}
      \alpha_{1}x_{1}\longrightarrow 0
      \quad\text{and}\quad
      \lob(\alpha_{1})+\lob(\alpha_{2})+\lob(\alpha_{3})\longrightarrow 0
    \end{equation*}
    because $\lob(0)=0$ and
    $\lob(\frac{\pi}{2}+\delta)+\lob(\frac{\pi}{2}-\delta)=0$.
    Now~\eqref{eq:limf} follows from 
    \begin{equation*}
      \alpha_{2}x_{2}+\alpha_{3}x_{3}-\tfrac{\pi}{2}(x_{2}+x_{3})
      \longrightarrow 0.
    \end{equation*}
    To see this, note that 
    \begin{equation*}
      \alpha_{2}x_{2}+\alpha_{3}x_{3}-\tfrac{\pi}{2}(x_{2}+x_{3})
      =-\tfrac{1}{2}\alpha_{1}(x_{2}+x_{3})
      +\tfrac{1}{2}(\alpha_{2}-\alpha_{3})(x_{2}-x_{3}).
    \end{equation*}
    The triangle inequalities imply $(x_{2}-x_{3})\rightarrow 0$, and
    using the sine rule one obtains
    $-\tfrac{1}{2}\alpha_{1}(x_{2}+x_{3})\rightarrow 0$ from
    $\lim_{\alpha\rightarrow 0}\alpha\log\sin\alpha=0$.
  \item As
    $(x_{1},x_{2},x_{3})\longrightarrow
    (\xbar_{1},\xbar_{2},\xbar_{3})\in\partial\!\Acal$,
    where $ e^{\xbar_{1}}=e^{\xbar_{2}}+e^{\xbar_{3}}, $
    \begin{equation*}
      f(x_{1},x_{2},x_{3})\longrightarrow \pi\xbar_{1}.
    \end{equation*}
    This follows from
    $(\alpha_{1},\alpha_{2},\alpha_{3})\longrightarrow (\pi,0,0)$.
    \qedhere
  \end{compactenum}
\end{proof}

The variational principle (see Theorem~\ref{thm:variational}) involves the
function $\Ecal_{\Theta,\Delta,\lambda}$ with $\Theta_{v_{\infty}}=0$,
$\Theta_{v}=2\pi$ for all other vertices, and
$u_{v_{\infty}}\rightarrow+\infty$. We denote this function by
$\Ecalbar^{v_{\infty}}_{\Delta,\lambda}$
(see Definition~\ref{def:Ecalbar}) and collect its relevant properties
(see Proposition~\ref{prop:Ecalbar}).

\begin{definition}[$\Ecalbar^{v_{\infty}}_{\Delta,\lambda}$]
  \label{def:Ecalbar}
  For a triangulation $\Delta$ of $(S_{g},V)$, a vertex
  $v_{\infty}\in V$, and $\lambda\in\R^{E_{\Delta}}$, let
  \begin{equation}
    \label{eq:Vo}
    \Vo=V\setminus\{v_{\infty}\}, 
  \end{equation}
  let $\Theta\in\R^{V}$ be defined by
  \begin{equation}
    \label{eq:ThetaEcalbar}
    \Theta_{v}=
    \begin{cases}
      0 & \text{if }v=v_{\infty}\\
      2\pi & \text{if } v\in \Vo, 
    \end{cases}
  \end{equation}
  and define
  \begin{gather}
    \notag
    \Ecalbar^{v_{\infty}}_{\Delta,\lambda}:
    \R^{\Vo}\longrightarrow\R\\
    \label{eq:Ecalbarlim}
    \Ecalbar^{v_{\infty}}_{\Delta,\lambda}(u)
    =\lim_{x\rightarrow+\infty}
    \Ecal_{\Theta,\Delta,\lambda}(u|_{u_{v_{\infty}}=x}),
  \end{gather}
  where for $u\in\R^{\Vo}$, we write $u|_{u_{v_{\infty}}=x}$ for the
  function in $\R^{V}$ with value $x$ at $v_{\infty}$ and agreeing
  with $u$ on $\Vo$.

\end{definition}

\begin{proposition}[properties of $\Ecalbar^{v_{\infty}}_{\Delta,\lambda}$]
  \label{prop:Ecalbar}
  (i) The limit in equation~\eqref{eq:Ecalbarlim} exists and is equal to
  \begin{multline}
    \label{eq:Ecalbar}
    \Ecalbar^{v_{\infty}}_{\Delta,\lambda}(u) = \sum_{t\in
      T_{\Deltilo}} 2f \Big( \frac{\lamtil_{e_{1}(t)}}{2},
    \frac{\lamtil_{e_{2}(t)}}{2}, \frac{\lamtil_{e_{3}(t)}}{2} \Big)
    -\pi\sum_{e\in E_{\Deltilo}}\lamtil_{e}
    \\
    -2\pi\sum_{v\in \Vo}(\log c_{v}(\Delta,\lambda)-u_{v}),
  \end{multline}
  where 
  \begin{equation}
    \label{eq:Deltillamtil}
    (\Deltil,\lamtil)
    =\Del(\Delta,\lambda,u|_{u_{v_{\infty}}=+\infty})
  \end{equation}
  (see Definition~\ref{def:Del2}), and $E_{\Deltilo}$ and $T_{\Deltilo}$
  are defined by~\eqref{eq:Eo} and~\eqref{eq:To}.
  
  (ii) The function $\Ecalbar^{v_{\infty}}_{\Delta,\lambda}$ is twice
  continuously differentiable and analytic in the interior of each
  Penner cell.

  (iii) The partial derivatives are
  \begin{equation}
    \label{eq:dEcalbar}
    \frac{\partial}{\partial u_{v}}\,
    \Ecalbar^{v_{\infty}}_{\Delta,\lambda}(u)
    =
    -\Thetatil_{v}+\pi(\deg_{2}(v)-\deg_{1}(v)+2),
  \end{equation}
  where $\Thetatil$ is defined as in Definition~\ref{def:realizable},
  equation~\eqref{eq:Thetatil}, and
  \begin{equation*}
    \deg_{1}(v)=\text{ edge-degree of $v$},
  \end{equation*}
  that is, the number of edges emanating from $v$ counted with
  multiplicity, and
  \begin{equation*}
    \deg_{2}(v)=\text{ triangle-degree of $v$},
  \end{equation*}
  that is, the number of triangles around $v$ counted with
  multiplicity.

  (iv) The second derivative is 
  \begin{equation}
   \label{eq:d2Ecalbar}
    D^{2}\Ecalbar^{v_{\infty}}_{\Delta,\lambda}\big|_{u}
    = 
    \frac{1}{4}\,\sum_{e\in E_{\Deltilo}}
    w_{e}(u)\,\big(du_{v_{1}(e)}-du_{v_{2}(e)}\big)^{2},
  \end{equation}
  where, if edge $e$ is not a boundary edge of $\Deltilo$,
  \begin{equation*}
    w_{e}(u)=\cot\alpha_{e}(u)+\cot\alpha_{e}'(u)
  \end{equation*}
  and $\alpha_{e}(u)$, $\alpha_{e}'(u)$ are the angles opposite $e$ in
  the piecewise euclidean metric defined in
  Definition~\ref{def:realizable} (r2b). If $e$ is a boundary edge,
  then $e$ has one or zero opposite angles and
  $w_{e}(u)=\cot\alpha_{e}(u)$ or $w_{e}(u)=0$, respectively.

  (v) The function $\Ecalbar^{v_{\infty}}_{\Delta,\lambda}$ is
  convex and satisfies the scaling relation
  \begin{equation}
    \label{eq:Ecalbarscale}
    \Ecalbar^{v_{\infty}}_{\Delta,\lambda}(u+h\,\mathbf{1}_{\Vo})
    =\Ecalbar^{v_{\infty}}_{\Delta,\lambda}(u)+2\pi\,h.
  \end{equation}
\end{proposition}

\begin{proof}
  Statement (i) follows from Lemma~\ref{lem:limEcal0}. By direct
  calculations, one obtains equations~\eqref{eq:dEcalbar}
  and~\eqref{eq:d2Ecalbar} in the interior of Penner cells. As for
  $\Hcal_{\Theta,\Delta}$, one finds that the first and second
  derivatives are continuous at the boundaries of Penner cells. This
  implies the differentiability statement (ii). 

  Statement (iii) follows from the corresponding properties of
  $\Ecal_{\Theta,\Delta,\lambda}$ (see Proposition~\ref{prop:Ecal}),
  because convexity and the scaling relation survive taking the
  limit~\eqref{eq:Ecalbarlim}. Note that
  \begin{equation*}
    \sum_{v\in V}\Theta_{v}=2\pi(|V|-1) 
  \end{equation*}
  due to~\eqref{eq:ThetaEcalbar}, and hence 
  \begin{equation*}
    |T_{\Deltil}|-|E_{\Deltil}|+\frac{1}{2\pi}\sum_{v\in V}\Theta_{v} 
    = |T_{\Deltil}|-|E_{\Deltil}|+|V|-1=1,
  \end{equation*}
  because $\Deltil$ is a triangulation of a sphere. Alternatively, one
  can also deduce statement (v) directly from
  equation~\eqref{eq:Ecalbar}.
\end{proof}

\begin{theorem}[variational principle for Problem~\ref{prob:realize2}]
  \label{thm:variational}
  Let $S\in\Teich_{0,n}$ be a complete finite area hyperbolic surface
  of genus $0$ with $n\geq 3$ cusps, and let $(\Delta,\lambda)$ be
  Penner coordinates for $S$, decorated with arbitrary horocycles.
  Let $\Vo=V\setminus\{v_{\infty}\}$ for some distinguished vertex
  $v_{\infty}\in V$. 
  \begin{compactenum}[(i)]
  \item If the function $\Ecalbar^{v_{\infty}}_{\Delta,\lambda}$ attains
    its minimum under the constraints
    \begin{equation}
      \label{eq:boxconstraint}
      u_{v} \geq \delta_{\Delta,\lambda}(v,v_{\infty}).
    \end{equation}
    at the point $u\in\Vo$, then $(\Deltil,\lamtil)$ defined
    by~\eqref{eq:Deltillamtil} are realizable coordinates with
    distinguished vertex $v_{\infty}$.
  \item Up to equivalence (see Definition~\ref{def:realizequiv}), all
    realizable coordinates with distinguished vertex $v_{\infty}$
    correspond to constrained minima of
    $\Ecalbar^{v_{\infty}}_{\Delta,\lambda}$ as in (i).
  \end{compactenum}
\end{theorem}

\begin{proof}
  We will show (i) and omit the proof of the converse statement (ii)
  because it is easier and no new ideas are required. So assume
  $\Ecalbar^{v_{\infty}}_{\Delta,\lambda}$ attains a minimum under the
  constraints~\eqref{eq:boxconstraint} at $u\in \R^{\Vo}$. We have to
  show conditions (r1) and (r2) of
  Definition~\ref{def:realizable}. Since condition (r1) obviously
  holds by construction, it remains to show (r2).
 
  First, note that the convex function
  $\Ecalbar^{v_{\infty}}_{\Delta,\lambda}$ attains a minimum under the
  constraints~\eqref{eq:boxconstraint} at $u\in \R^{\Vo}$ if
  and only if for all $v\in\Vo$:
  \begin{alignat}{2}
    \label{eq:minimum_int}
    \frac{\partial}{\partial u_{v}}\,
    \Ecalbar^{v_{\infty}}_{\Delta,\lambda}(u) 
    &=0 
    &\qquad
    &\text{if\; }u_{v}>\delta_{\Delta,\lambda}(v,v_{\infty})\\
    \label{eq:minimum_bdy}
    \frac{\partial}{\partial u_{v}}\,
    \Ecalbar^{v_{\infty}}_{\Delta,\lambda}(u) 
    &\geq 0 
    &\qquad
    &\text{if\; }u_{v}=\delta_{\Delta,\lambda}(v,v_{\infty})
  \end{alignat}
  The scaling relation~\eqref{eq:Ecalbarscale} implies that if a
  constrained minimum is attained at~$u$ then $u$ satisfies at least
  one constraint~\eqref{eq:boxconstraint} with equality. The vertices
  $v\in\Vo$ for which the constraint~\eqref{eq:boxconstraint} is
  satisfied with equality are precisely the vertices
  adjacent to $v_{\infty}$ in $\Deltil$
  (see Proposition~\ref{prop:punctured}). Now let $v\in V$ be a vertex
  adjacent to $v_{\infty}$. By equation~\eqref{eq:dEcalbar} and
  inequality~\eqref{eq:minimum_bdy}, we have
  \begin{equation}
    \label{eq:min_bdy_subst}
    -\Thetatil_{v}+\pi(\deg_{2}(v)-\deg_{1}(v)+2)\geq 0.
  \end{equation}
  Consider two cases separately:
  \begin{compactenum}[(a)]
  \item $\deg_{2}(v)=0$: In this case $\Thetatil_{v}=0$ because there
    are no triangles $t\in T_{\Deltilo}$ incident with
    $v$. Inequality~\eqref{eq:min_bdy_subst} implies
    \begin{equation*}
      \deg_{1}(v)\leq 2.
    \end{equation*}
    Using the following two observations, one deduces that the cell
    complex $\Deltilo=(\Vo, E_{\Deltilo},T_{\Deltilo})$ is a linear
    graph:
    \begin{compactenum}[(1)]
    \item Any vertex $v'\not=v_{\infty}$ adjacent to $v$ also
      satisfies $\deg_{2}(v)=0$.
    \item The cell complex $\Deltilo$ is connected because its complement
      in the sphere~$S_{0}$ is an open disk.
    \end{compactenum}
    This proves condition (r2a) of Definition~\ref{def:realizable}.
  \item $\deg_{2}(v)>0$: In this case $\Thetatil_{v}>0$, so
    inequality~\eqref{eq:min_bdy_subst} implies
    \begin{equation*}
      \deg_{1}(v)<\deg_{2}(v)+2.
    \end{equation*}
    On the other hand, because $T_{\Deltilo}$ does not contain all
    triangles of $T_{\Deltil}$ incident with $v$, we have
    \begin{equation*}
      \deg_{1}(v)\geq\deg_{2}(v)+1,
    \end{equation*}
    and therefore
    \begin{equation}
      \label{eq:deg1}
      \deg_{1}(v)=\deg_{2}(v)+1.
    \end{equation}
    Because the complement of $\Deltilo$ in the sphere~$S_{0}$ is an
    open disk, this implies that the cell complex $\Deltilo$ is a
    triangulation of a closed disk. With~\eqref{eq:deg1},
    inequality~\eqref{eq:min_bdy_subst}
    implies~\eqref{eq:Thetatilineq}, and~\eqref{eq:minimum_int}
    implies~\eqref{eq:Thetatileq}.  This proves condition (r2b).
  \end{compactenum}
  This concludes the proof of (i).
\end{proof}

\section{The differentiability lemma}
\label{sec:difflemma}

In this section we treat Lemma~\ref{lem:Hsf}, which proves the
well-definedness and differentiability statement of Definition and
Proposition~\ref{defprop:Hcal}.

\begin{lemma}
  \label{lem:Hsf}
  Suppose $\Delta_{1}$ and $\Delta_{2}$ are both Delaunay
  triangulations for the decorated surface with Penner coordinates
  $(\Delta,\lambda^{*})$, and let
  $\tau_{12}=\tau_{\Delta_{1},\Delta_{2}}$ be the chart transition
  function $\R^{E_{\Delta_{1}}}\rightarrow\R^{E_{\Delta_{2}}}$ mapping
  Penner coordinates with respect to $\Delta_{1}$ to Penner
  coordinates with respect to $\Delta_{2}$. Then the function values
  and the first and second derivatives of
  $\Hsf_{\Theta}(\Delta_{1},\,\cdot\,)$ and
  $\Hsf_{\Theta}(\Delta_{2},\tau_{12}(\,\cdot\,))$ at $\lambda^{*}$
  are equal:
  \begin{align}
    \label{eq:Hsfflip}
    \Hsf_{\Theta}(\Delta_{1},\lambda^{*})
    &=\Hsf_{\Theta}\big(\Delta_{2},\tau_{12}(\lambda^{*})\big),\\
    \label{eq:dHsfflip}
    d\Hsf_{\Theta}(\Delta_{1},\,\cdot\,)\big|_{\lambda^{*}}
    &=d\big(\Hsf_{\Theta}(\Delta_{2},\tau_{12}(\,\cdot\,))\big)
      \big|_{\lambda^{*}},\\
    \label{eq:d2Hsfflip}
    D^{2}\Hsf_{\Theta}(\Delta_{1},\,\cdot\,)\big|_{\lambda^{*}}
    &=D^{2}\big(\Hsf_{\Theta}(\Delta_{2},\tau_{12}(\,\cdot\,))\big)
      \big|_{\lambda^{*}}.
  \end{align}
\end{lemma}

\begin{remarks}
  \label{rem:difflem}
  (i) It is easy to check numerically that the analogous equations for
  higher derivatives are in general false. In particular,
  \begin{equation*}
    D^{3}\Hsf_{\Theta}(\Delta_{1},\,\cdot\,)\big|_{\lambda^{*}}
    \not=D^{3}\big(\Hsf_{\Theta}(\Delta_{2},\tau_{12}(\,\cdot\,))
      \big)\big|_{\lambda^{*}},
  \end{equation*}
  so the third derivative of the function $\Hcal_{\Theta}$ on
  $\decTeich_{g,n}$ is generally discontinuous at the boundaries of
  Penner cells.

  (ii) The proof of equation~\eqref{eq:d2Hsfflip} given in this
  section consists of a straightforward but unilluminating
  calculation. A more conceptual argument would be desirable.

  (iii) In the variational principles of
  Theorems~\ref{thm:variational} and~\ref{thm:variational2}, only the
  function $\Ecal_{\Theta,\Delta,\lambda}(u)$ plays a role, which is
  the restriction of $\Hcal_{\Theta,\Delta}(\lambda)$ to one fiber of
  the decorated \Tm{} space $\decTeich_{g,n}$. In the context of
  realization and discrete uniformization theorems, it would be enough
  to show that $\Ecal_{\Theta,\Delta,\lambda}(u)$ is twice
  continuously differentiable. In view of
  Theorem~\ref{thm:idealeucdel} and Remark~\ref{rem:d2Esf}, this
  follows from Rippa's Minimal Roughness Theorem~\cite{rippa90} for
  the PL Dirichlet energy. (This is also proved by an unilluminating
  calculation.) Lemma~\ref{lem:Hsf} shows that the function
  $\Hcal_{\Theta,\Delta}$ and hence the function $\Hcal_{\Theta}$ of
  Corollary~\ref{cor:Hcal} is twice continuously differentiable on the
  whole decorated \Tm{} space. We believe this is of independent
  interest.
\end{remarks}

\noindent
The rest of this section is devoted to the proof of
Lemma~\ref{lem:Hsf}.

\paragraph{1.~Reduction to $\Theta=0$.}
Since the total length of the decorating horocycle at a vertex $v$
does not depend on the triangulation, we have
\begin{equation*}
  c_{v}(\Delta_{1},\lambda)=
  c_{v}(\Delta_{2},\tau_{12}(\lambda))
\end{equation*}
and hence trivially also
\begin{align}
  dc_{v}(\Delta_{1},\,\cdot\,)\big|_{\lambda}
  &=dc_{v}(\Delta_{2},\tau_{12}(\,\cdot\,))\big|_{\lambda},\\
  D^{2}c_{v}(\Delta_{1},\,\cdot\,)\big|_{\lambda}
  &=D^{2}c_{v}(\Delta_{2},\tau_{12}(\,\cdot\,))\big|_{\lambda}
\end{align}
for all ideal triangulations $\Delta_{1}$, $\Delta_{2}$ and all
$\lambda\in\R^{E_{\Delta_{1}}}$.  To prove Lemma~\ref{lem:Hsf}, it
is therefore enough to consider the function~$\Hsf_{0}$ with
$\Theta=0$.

\paragraph{2.~Reduction to a single edge flip.}
Without loss of generality, we may assume that $\Delta_{1}$ and
$\Delta_{2}$ differ by a single edge flip. Indeed, any two Delaunay
triangulations for the same decorated surface are related by a finite
sequence of flips of nonessential edges, and
equations~\eqref{eq:Hsfflip}--\eqref{eq:d2Hsfflip} have the necessary
transitivity property. To be more specific, assume $\Delta_{1}$,
$\Delta_{2}$, $\Delta_{3}$ are three ideal triangulations. To
abbreviate, we write $\Hsf_{i}$ for
$\Hsf_{\Theta}(\Delta_{i},\,\cdot\,)$ and $\tau_{ij}$ for
$\tau_{\Delta_{i},\Delta_{j}}$. Then, by a straightforward application
of the chain rules for first and second derivatives, the
equations
\begin{equation*}
\begin{aligned}
  \Hsf_{1}(\lambda^{*})&=\Hsf_{2}\circ\tau_{12}(\lambda^{*}),\\
  d\Hsf_{1}\big|_{\lambda^{*}}&=d(\Hsf_{2}\circ\tau_{12})\big|_{\lambda^{*}},\\
  D^{2}\Hsf_{1}\big|_{\lambda^{*}}&=D^{2}(\Hsf_{2}\circ\tau_{12})\big|_{\lambda^{*}},
\end{aligned}
\quad\text{and}\quad
\begin{aligned}
  \Hsf_{2}(\tau_{12}(\lambda^{*}))&=\Hsf_{3}\circ\tau_{23}(\tau_{12}(\lambda^{*})),\\
  d\Hsf_{2}\big|_{\tau_{12}(\lambda^{*})}&=d(\Hsf_{3}\circ\tau_{23})\big|_{\tau_{12}(\lambda^{*})},\\
  \quad
  D^{2}\Hsf_{2}\big|_{\tau_{12}(\lambda^{*})}&=D^{2}(\Hsf_{3}\circ\tau_{23})\big|_{\tau_{12}(\lambda^{*})}
\end{aligned}
\end{equation*}
imply
\begin{align*}
  \Hsf_{1}(\lambda^{*})&=\Hsf_{3}\circ\tau_{13}(\lambda^{*}),\\
  d\Hsf_{1}\big|_{\lambda^{*}}&=d(\Hsf_{3}\circ\tau_{13})\big|_{\lambda^{*}},\\
  D^{2}\Hsf_{1}\big|_{\lambda^{*}}&=D^{2}(\Hsf_{3}\circ\tau_{13})\big|_{\lambda^{*}}.
\end{align*}

\paragraph{3. Notation.} In the following we assume that $\Delta_{2}$
is the result of flipping edge~$e$ of $\Delta_{1}$, which is replaced
by edge $f$ in $\Delta_{2}$.
% , so
% \begin{equation*}
%   E_{\Delta_{1}}\setminus\{e\}=E_{\Delta_{2}}\setminus\{f\}.
% \end{equation*}
Let $a,b,c,d\in E_{\Delta_{1}}\cap E_{\Delta_{2}}$ be the adjacent edges of
$e$ and $f$ as in Figure~\ref{fig:ptolemy}, and let~$\ell$ be
defined by~\eqref{eq:ell}. By Lemma~\ref{lem:triang} and
Theorem~\ref{thm:idealeucdel}, the euclidean triangles with side
lengths $\ell_{a},\ell_{b},\ell_{e}$ and
$\ell_{e},\ell_{c},\ell_{d}$ form a cyclic quadrilateral as shown in
Figure~\ref{fig:quadnotation}.
\begin{figure}
  \labellist
  \small\hair 2pt
  \pinlabel {$\ell_{a}$} [t] at 64 23
  \pinlabel {$\ell_{b}$} [l] at 105 70
  \pinlabel {$\ell_{c}$} [b] at 55 109
  \pinlabel {$\ell_{d}$} [r] at 16 62
  \pinlabel {$\ell_{e}$} [br] at 46 57
  \pinlabel {$\ell_{f}$} [bl] at 74 52
  \pinlabel {$\alpha$} [ ] at 92 98
  \pinlabel {$\alpha$} [ ] at 22 86
  \pinlabel {$\beta$} [ ] at 31 31
  \pinlabel {$\beta$} [ ] at 34 98
  \pinlabel {$\gamma$} [ ] at 22 41
  \pinlabel {$\gamma$} [ ] at 102 41
  \pinlabel {$\delta$} [ ] at 80 106
  \pinlabel {$\delta$} [ ] at 92 30.5
  \endlabellist
  \centering
  \includegraphics[scale=0.9]{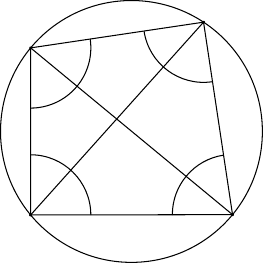}
  \caption{Lengths and angles in a euclidean cyclic quadrilateral.}
  \label{fig:quadnotation}
\end{figure}
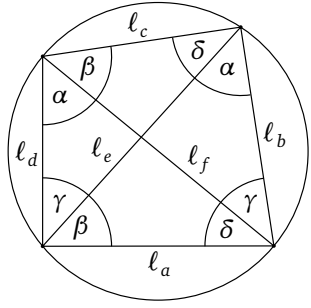
By Ptolemy's theorem of euclidean geometry, $\ell_{f}$ is the length
of the other diagonal. 

\paragraph{4. Equality of function values.} To show
equation~\eqref{eq:Hsfflip}, we use the notation of
Figure~\ref{fig:quadnotation} for the angles. Writing
\begin{equation}
  \label{eq:Vtet}
  \Vtet(\alpha_{1},\alpha_{2},\alpha_{3})
  =\lob(\alpha_{1})+\lob(\alpha_{2})+\lob(\alpha_{3}),
\end{equation}
we obtain from equation~\eqref{eq:Hsf}
\begin{multline*}
  \Hsf_{0}(\Delta_{1},\lambda^{*})
  -\Hsf_{0}\big(\Delta_{2},\tau_{12}(\lambda^{*})\big)=
  \\
  2\Vtet(\alpha,\beta,{\gamma}+{\delta})
  +2\Vtet(\gamma,\delta,{\alpha}+{\beta})
  +({\alpha}+{\beta}+{\gamma}+{\delta}-\pi)\lambda^{*}_{e}\\
  -2\Vtet({\beta},{\gamma},\alpha+\delta)
  -2\Vtet({\delta},{\alpha},\beta+\gamma)
  -(\alpha+\beta+\gamma+\delta-\pi)\lambda^{*}_{f},
\end{multline*}
which is zero because 
\begin{equation}
  \label{eq:albegadesum}
  \alpha+\beta+\gamma+\delta=\pi
\end{equation}
and because
equation~\eqref{eq:Vtet} is Milnor's formula~\cite{milnor82} for the
volume $\Vtet$ of an ideal tetrahedron with dihedral angles
$\alpha_{1}$, $\alpha_{2}$, $\alpha_{3}$ as shown in
Figure~\ref{fig:idealtet}.  So
\begin{equation*}
  \Vtet(\alpha,\beta,{\gamma}+{\delta})
  +\Vtet(\gamma,\delta,{\alpha}+{\beta})\quad
  \text{and}\quad
  \Vtet({\beta},{\gamma},\alpha+\delta)
  +\Vtet({\delta},{\alpha},\beta+\gamma)
\end{equation*}
are two ways of writing the volume of an ideal quadrilateral pyramid
as the sum of the volumes of two tetrahedra.

\paragraph{5. Equality of first derivatives.} To show
equation~\eqref{eq:dHsfflip}, note that the chart transition
function~$\tau_{12}$ changes $\lambda^{*}_{e}$ to
$\lambda^{*}_{f}$ as determined by Ptolemy's
relation~\eqref{eq:ptolemy} and leaves the values of $\lambda^{*}$ for
all other edges unchanged. Now consider the partial derivatives of
both sides of equation~\eqref{eq:Hsfflip} at
$\lambda^{*}$. From~\eqref{eq:df} and~\eqref{eq:Hsf} one obtains by a
straightforward calculation
\begin{align}
  &\frac{\partial\Hsf_{0}(\Delta_{1},\,\cdot\,)}{\partial \lambda_{e}}
    \Big|_{\lambda^{*}}
    ={\alpha}+{\beta}+{\gamma}+{\delta}-\pi=0,\\
  \intertext{and similarly}
  \label{eq:dH0dlambdaf}
  &\frac{\partial\Hsf_{0}(\Delta_{2},\,\cdot\,)}{\partial \lambda_{f}}
    \Big|_{\tau_{12}(\lambda^{*})}
    =\delta+\alpha+\beta+\gamma-\pi=0,\\
  \intertext{which implies}
  &\frac{\partial\Hsf_{0}(\Delta_{2},\tau_{12}(\,\cdot\,))}{\partial
    \lambda_{e}}
    \Big|_{\lambda^{*}}
    =0,
\end{align}
and hence equality of the partial derivatives with respect to
$\lambda_{e}$. For the partial derivative with respect to $\lambda_{a}$ one obtains
\begin{equation}
  \frac{\partial\Hsf_{0}(\Delta_{1},\,\cdot\,)}{\partial \lambda_{a}}
  \Big|_{\lambda^{*}}
  =\alpha
  =\frac{\partial\Hsf_{0}(\Delta_{2},\tau_{12}(\,\cdot\,))}{\partial
    \lambda_{a}}
  \Big|_{\lambda^{*}}
\end{equation}
and similarly for $\lambda_{b}$, $\lambda_{c}$, $\lambda_{d}$. For all other edges
$\varepsilon\in E_{\Delta_{1}}\cap E_{\Delta_{2}}$, the difference of
partial derivatives is zero because all terms depending on
$\lambda_{\varepsilon}$ in the difference
$\Hsf_{0}(\Delta_{1},\lambda)
-\Hsf_{0}(\Delta_{2},\tau_{12}(\lambda))$ cancel. This
proves~\eqref{eq:dHsfflip}.

\paragraph{5. Some useful identities.} In the calculation proving
equation~\eqref{eq:d2Hsfflip}, we will use the identities
\begin{equation}
  \label{eq:dlam}
  \begin{split}
    -d\lambda_{a}-d\lambda_{c}+d\lambda_{e}+d\lambda_{f} &=
    \frac{\sin\beta\sin\delta}{\sin\alpha\sin\gamma+\sin\beta\sin\delta}\,
    (-d\lambda_{a}+d\lambda_{b}-d\lambda_{c}+d\lambda_{d}),\\
    d\lambda_{b}+d\lambda_{d}-d\lambda_{e}-d\lambda_{f} &=
    \frac{\sin\alpha\sin\gamma}{\sin\alpha\sin\gamma+\sin\beta\sin\delta}\,
    (-d\lambda_{a}+d\lambda_{b}-d\lambda_{c}+d\lambda_{d}),
  \end{split}
\end{equation}
which are valid at $\lambda^{*}$, as well as the simple trigonometric identities
\begin{equation}
  \label{eq:cotidentity}
  \cot x\pm \cot y=\frac{\sin(\pm x+y)}{\sin x\sin y}.
\end{equation}

To see~\eqref{eq:dlam}, take derivatives on both sides
of of Ptolemy's relation
\begin{equation*}
  e^{\frac{1}{2}(\lambda_{e}+\lambda_{f})}
  =e^{\frac{1}{2}(\lambda_{a}+\lambda_{c})}
  +e^{\frac{1}{2}(\lambda_{b}+\lambda_{d})}
\end{equation*}
to obtain
\begin{align*}
  d\lambda_{e}+d\lambda_{f}
  &=\frac{e^{\frac{1}{2}(\lambda_{a}+\lambda_{c})}}{e^{\frac{1}{2}(\lambda_{a}+\lambda_{c})}+e^{\frac{1}{2}(\lambda_{b}+\lambda_{d})}}
    \,(d\lambda_{a}+d\lambda_{c})
    +\frac{e^{\frac{1}{2}(\lambda_{b}+\lambda_{d})}}{e^{\frac{1}{2}(\lambda_{a}+\lambda_{c})}+e^{\frac{1}{2}(\lambda_{b}+\lambda_{d})}}
    \,(d\lambda_{b}+d\lambda_{d})\\
  &=\frac{1}{1+e^{\frac{1}{2}(-\lambda_{a}+\lambda_{b}-\lambda_{c}+\lambda_{d})}}
    \,(d\lambda_{a}+d\lambda_{c})
    +\frac{1}{1+e^{\frac{1}{2}(\lambda_{a}-\lambda_{b}+\lambda_{c}-\lambda_{d})}}
    \,(d\lambda_{b}+d\lambda_{d}), 
\end{align*}
then by the law of sines
\begin{multline}
  \label{eq:dptolemy}
  d\lambda_{e}+d\lambda_{f}
  =\frac{\sin\alpha\sin\gamma}{\sin\alpha\sin\gamma+\sin\beta\sin\delta}
  \,(d\lambda_{a}+d\lambda_{c})\\
  +\frac{\sin\beta\sin\delta}{\sin\alpha\sin\gamma+\sin\beta\sin\delta}
  \,(d\lambda_{b}+d\lambda_{d}).
\end{multline}
and finally the identities~\eqref{eq:dlam}.

\paragraph{6. Equality of second derivatives.} To show
equation~\eqref{eq:d2Hsfflip}, first consider the right hand
side. The chain rule for second derivatives says
\begin{multline}
    D^{2}\big(\Hsf_{0}(\Delta_{2},\tau_{12}(\,\cdot\,))\big)
    \big|_{\lambda^{*}}(v,w)=
    D^{2}\Hsf_{0}(\Delta_{2},\,\cdot\,)\big|_{\tau_{12}(\lambda^{*})}
    \big(d\tau_{12}|_{\lambda^{*}}(v),d\tau_{12}|_{\lambda^{*}}(w)\big)\\
    + \underbrace{d\Hsf_{0}(\Delta_{2},\,\cdot\,)\big|_{\tau_{12}(\lambda^{*})}
    \big(D^{2}\tau_{12}\big|_{\lambda^{*}}(v,w)\big)}_{=0},
\end{multline}
where the term involving $d\Hsf_{0}(\Delta_{2},\,\cdot\,)$ vanishes due to
equation~\eqref{eq:dH0dlambdaf}, because $D^{2}\tau_{12}$ only has a component in the direction of
$\tfrac{\partial}{\partial\lambda_{f}}$.

Now use~\eqref{eq:d2f}, \eqref{eq:Hsf}, and~\eqref{eq:albegadesum} to
obtain \begingroup \allowdisplaybreaks
\begin{align*}
  2\Big(
  & D^{2}\Hsf_{0}(\Delta_{1},\,\cdot\,)
    -D^{2}\big(\Hsf_{0}(\Delta_{2},\tau_{12}(\,\cdot\,))
    \big)\Big)\Big|_{\lambda^{*}}\\
  =\;
  &
    \cot\alpha\,\big((d\lambda_{b}-d\lambda_{e})^{2}-(d\lambda_{f}-d\lambda_{d})^{2}\big)
    +
    \cot\beta\,\big((d\lambda_{e}-d\lambda_{a})^{2}-(d\lambda_{c}-d\lambda_{f})^{2}\big)\\
  & +
    \cot\gamma\,\big((d\lambda_{d}-d\lambda_{e})^{2}-(d\lambda_{f}-d\lambda_{b})^{2}\big)
    +
    \cot\delta\,\big((d\lambda_{e}-d\lambda_{c})^{2}-(d\lambda_{a}-d\lambda_{f})^{2}\big)\\
  & +
    \cot(\alpha+\beta)\big((d\lambda_{c}-d\lambda_{d})^{2}-(d\lambda_{a}-d\lambda_{b})^{2}\big)\\
  & -
    \cot(\beta+\gamma)\big((d\lambda_{d}-d\lambda_{a})^{2}-(d\lambda_{b}-d\lambda_{c})^{2}\big)\\
  =\;
  & \cot\alpha\,
    (d\lambda_{b}+d\lambda_{d}-d\lambda_{e}-d\lambda_{f})
    (d\lambda_{b}-d\lambda_{d}-d\lambda_{e}+d\lambda_{f})\\
  & +
    \cot\beta\,
    (-d\lambda_{a}-d\lambda_{c}+d\lambda_{e}+d\lambda_{f})
    (-d\lambda_{a}+d\lambda_{c}+d\lambda_{e}-d\lambda_{f})\\  
  & +
    \cot\gamma\,
    (d\lambda_{b}+d\lambda_{d}-d\lambda_{e}-d\lambda_{f})
    (-d\lambda_{b}+d\lambda_{d}-d\lambda_{e}+d\lambda_{f})\\  
  & +
    \cot\delta\,
    (-d\lambda_{a}-d\lambda_{c}+d\lambda_{e}+d\lambda_{f})
    (d\lambda_{a}-d\lambda_{c}+d\lambda_{e}-d\lambda_{f})\\  
  & +
    \cot(\alpha+\beta)
    (d\lambda_{a}-d\lambda_{b}+d\lambda_{c}-d\lambda_{d})
    (-d\lambda_{a}+d\lambda_{b}+d\lambda_{c}-d\lambda_{d})\\  
  & -
    \cot(\beta+\gamma)
    (-d\lambda_{a}+d\lambda_{b}-d\lambda_{c}+d\lambda_{d})
    (-d\lambda_{a}-d\lambda_{b}+d\lambda_{c}+d\lambda_{d})\\
  =\;
  &
    (d\lambda_{b}+d\lambda_{d}-d\lambda_{e}-d\lambda_{f})\\
  & \qquad
    \big(
    (\cot\alpha-\cot\gamma)(d\lambda_{b}-d\lambda_{d})
    +
    (\cot\alpha+\cot\gamma)(-d\lambda_{e}+d\lambda_{f})
    \big)\\
  & +
    (-d\lambda_{a}-d\lambda_{c}+d\lambda_{e}+d\lambda_{f})\\
  & \hphantom{+}\qquad
    \big(
    (\cot\beta-\cot\delta)(-d\lambda_{a}+d\lambda_{c})
    +
    (\cot\beta+\cot\delta)(d\lambda_{e}-d\lambda_{f})
    \big)\\
  & +
    (d\lambda_{a}-d\lambda_{b}+d\lambda_{c}-d\lambda_{d})\\
  & \hphantom{+}\qquad
    \big(
    (\cot(\alpha+\beta)-\cot(\beta+\gamma))(d\lambda_{b}-d\lambda_{d})\\
  & \qquad\quad +
    (\cot(\alpha+\beta)+\cot(\beta+\gamma))(-d\lambda_{a}+d\lambda_{c})
    \big)\\
  \overset{5.}{=}\;
  & (-d\lambda_{a}+d\lambda_{b}-d\lambda_{c}+d\lambda_{d})\\
  & \Bigg(
    \frac{1}{\sin\alpha\sin\gamma+\sin\beta\sin\delta}
    \big(
    \sin(-\alpha+\gamma)(d\lambda_{b}-d\lambda_{d})
    + \sin(\alpha+\gamma)(-d\lambda_{e}+d\lambda_{f})\\
  & \qquad\qquad\qquad\qquad\qquad
    + \sin(-\beta+\delta)(-d\lambda_{a}+d\lambda_{c})
    + \sin(\beta+\delta)(d\lambda_{e}-d\lambda_{f})\big)\\
  \stepcounter{equation}\tag{\theequation}\label{eq:d2product}
  & 
    -\frac{1}{\sin(\alpha+\beta)\sin(\beta+\gamma)}
    \big(
    \sin(-\alpha+\gamma)(d\lambda_{b}-d\lambda_{d})
    + \sin(-\beta+\delta)(-d\lambda_{a}+d\lambda_{c})
    \big)
    \Bigg)\\
  = & \;0.
\end{align*}
\endgroup

To see equality ``$\overset{5.}{=}$'', use the
identities~\eqref{eq:dlam} and~\eqref{eq:cotidentity}.
To see the last equality, note that~\eqref{eq:albegadesum} implies
\begin{equation}
  \label{eq:sinsum}
  \sin(\alpha+\beta)\sin(\beta+\gamma)
  = \sin\alpha\sin\gamma+\sin\beta\sin\delta.
\end{equation}
This proves equation~\eqref{eq:d2Hsfflip} and completes the proof of
Lemma~\ref{lem:Hsf}.

\begin{remark}
  In connection with Remark~\ref{rem:difflem}~(iii), it is curious to note
  that for a vertical tangent vector in $T\decTeich_{g,n}$, that is, a
  vector tangent to the fiber over a point in $\Teich_{g,n}$, the first factor in~\eqref{eq:d2product},
  \begin{equation*}
    -d\lambda_{a}+d\lambda_{b}-d\lambda_{c}+d\lambda_{d},
  \end{equation*}
  vanishes as well.
\end{remark}

\section{Proof of Theorem~\ref{thm:rivin}}
\label{sec:proof}

In this section, we prove Theorem~\ref{thm:rivin} using the
variational principle of Theorem~\ref{thm:variational}. By
Propositions~\ref{prop:poly2realizable}
and~\ref{prop:realizable2poly}, Theorem~\ref{thm:rivin} is equivalent
to the following statement about the existence and uniqueness of
realizable coordinate:

\begin{proposition}
  \label{prop:realizable}
  Problem~\ref{prob:realize2} has a unique solution up to equivalence
  (see Definition~\ref{def:realizequiv}). 
\end{proposition}

Proposition~\ref{prop:realizable} is in turn equivalent to the
following statement about the unique solvability of an optimization
problem with bounds constraints (see Theorem~\ref{thm:variational}):

\begin{proposition}
  \label{prop:minimum}
  Let $\Delta$ be a triangulation of $(S_{0},V)$, the oriented surface
  of genus~$0$ with $n=|V|\geq 3$ marked points, let
  $\lambda\in\R^{E_{\Delta}}$, let $v_{\infty}\in V$, and let~$\Vo$ be
  defined by~\eqref{eq:Vo}. Then there exists a unique solution to the
  minimization problem
  \begin{equation}
    \label{eq:optprob}
    \begin{aligned}
      &\text{minimize\ \;
          $\Ecalbar^{v_{\infty}}_{\Delta,\lambda}(u)$\ \; for\ \;$u\in\R^{\Vo}$,}\\
      &\text{subject to the bounds constraints~\eqref{eq:boxconstraint}}
    \end{aligned}
  \end{equation}
\end{proposition}

The rest of this section is concerned with proving
Proposition~\ref{prop:minimum}, first the uniqueness statement, then
the existence statement. 

\paragraph{Uniqueness.} Assume the minimization
problem~\eqref{eq:optprob} has a solution. To show the solution is
unique, assume $u\in\R^{\Vo}$ solves~\eqref{eq:optprob} and let
$(\Deltil,\lamtil)$ be defined by~\eqref{eq:Deltillamtil}. By
Theorem~\ref{thm:variational}, $(\Deltil,\lamtil)$ are realizable
coordinates with distinguished vertex $v_{\infty}$. 

Either the subcomplex $\Deltilo$ of cells not incident with
$v_{\infty}$ is a linear graph. In this case all
constraints~\eqref{eq:boxconstraint} are satisfied with equality,
which determines~$u$ uniquely.

Or $\Deltilo$ is a triangulation of a closed disk. Then the second
derivative of $\Ecalbar^{v_{\infty}}_{\Delta,\lambda}$ at $u$ is
positive semidefinite with one-dimensional kernel spanned by
$\mathbf{1}_{\Vo}$. This follows from~\eqref{eq:Ecalbar} and 
the convexity of $f$ (see Proposition\ref{prop:f}~(iii)). Together
with~\eqref{eq:minimum_int} and~\eqref{eq:minimum_bdy}, this implies
that 
\begin{equation*}
  \Ecalbar^{v_{\infty}}_{\Delta,\lambda}(u')>
  \Ecalbar^{v_{\infty}}_{\Delta,\lambda}(u)
\end{equation*}
for any $u'\in\R^{\Vo}\setminus\{u\}$ satisfying the
constraints~\eqref{eq:boxconstraint}.

\paragraph{Existence.} To show that the continuous function
$\Ecalbar^{v_{\infty}}_{\Delta,\lambda}$ attains its minimum on the
closed subset 
\begin{equation*}
  D=\big\{u\in\R^{\Vo}\,\big|\,\;
  u\text{ satisfies the bounds 
    constraints~\eqref{eq:boxconstraint}}\big\}\subseteq\R^{\Vo},
\end{equation*}
it is enough to show that every unbounded sequence
$(u_{n})$ in $D$ has a subsequence~$(u_{n_{k}})$ with
$\Ecalbar^{v_{\infty}}_{\Delta,\lambda}(u_{n_{k}})\rightarrow+\infty$. 

So let $(u_{n})$ be an unbounded sequence in $D$. Note that $(u_{n})$
is bounded from below by the
constraints~\eqref{eq:boxconstraint}. Hence, after taking a
subsequence if necessary, we may assume that for every $v\in\Vo$
either $u_{n}(v)$ converges to a finite limit, or
$u_{n}(v)\rightarrow+\infty$. Since $u_{n}(v)\rightarrow+\infty$ for
at least one $v\in\Vo$ and $\Theta_{v}=2\pi$,
Corollary~\ref{cor:limEcal} implies
\begin{equation*}
  \lim_{n}\Ecalbar^{v_{\infty}}_{\Delta,\lambda}(u_{n})=+\infty.
\end{equation*}
This concludes the proof of Proposition~\ref{prop:minimum}, and hence
of Theorem~\ref{thm:rivin}.

\section{Discrete conformal equivalence and the uniformization
  spheres}
\label{sec:uniform}

In this section we recall the basic definitions of discrete conformal
equivalence, and we discuss the equivalence of Rivin's
Theorem~\ref{thm:rivin} and the discrete uniformization theorem for
spheres (see Theorem~\ref{thm:unisphere}). The relation of discrete
conformal equivalence was first defined for triangulated piecewise
euclidean surfaces. As explained in
Definition~\ref{def:triangpweucsurf}, we use the notation
$(\Delta,\ell)$ for the triangulated piecewise euclidean surface with
triangulation $\Delta$ and edge lengths $\ell\in\R_{>0}^{E_{\Delta}}$.

\begin{definition}[discrete conformal equivalence of triangulated
  surfaces]
  \label{def:dce1}
  The triangulated piecewise euclidean surfaces $(\Delta,\ell)$ and
  $(\Delta,\elltil)$ are \emph{discretely conformally equivalent} if
  there is a function $u\in\R^{V_{\Delta}}$ such that for every edge
  $e\in E_{\Delta}$,
  \begin{equation}
    \label{eq:udc}
    \elltil(e)
    =e^{\frac{1}{2}(u_{v_{1}(e)}+u_{v_{2}(e)})}\ell(e).
  \end{equation}
\end{definition}

This definition is due to Luo~\cite{luo04}. It has the following
interpretation in terms of hyperbolic geometry:

\begin{proposition}[see {\cite[Theorem~5.1.2]{bobenko15}}]
  \label{prop:dcehyp}
  On a triangulated piecewise euclidean surface, one obtains a
  complete hyperbolic metric of finite area by equipping every
  triangle with the hyperbolic Klein metric induced by its
  circumcircle. Then the following statements are equivalent:
  \begin{compactenum}[(i)]
  \item The triangulated piecewise euclidean surfaces $(\Delta,\ell)$
    and $(\Deltil,\elltil)$ are discretely conformally equivalent.
  \item The triangulated piecewise euclidean surfaces $(\Delta,\ell)$
    and $(\Deltil,\elltil)$ are isometric with respect to the
    induced hyperbolic metrics.
  \end{compactenum}
\end{proposition}

The induced hyperbolic
metric has Penner coordinates $(\Delta,\lambda)$, where $\lambda$ and
$\ell$ are related by~\eqref{eq:ell}. Proposition~\ref{prop:dcehyp} can also be
seen by considering decorated ideal tetrahedra as shown in
Figure~\ref{fig:idealtet}: The projection form the point $v_{\infty}$
maps the euclidean triangle $A_{1}A_{2}A_{3}$ in the horosphere
centered at $v_{\infty}$ to the ideal triangle $v_{1}v_{2}v_{3}$. The
hyperbolic Klein metric induced on the euclidean triangle
$A_{1}A_{2}A_{3}$ by its circumcircle is the pullback of the
hyperbolic metric of the ideal triangle $v_{1}v_{2}v_{3}$.

Note that Definition~\ref{def:dce1} requires the triangulations of
both surfaces to be equal. Discrete conformal mapping problems based
on this notion of discrete conformal equivalence can be solved using
the variational principles introduced in~\cite{bobenko15}---provided a
solution exists. The variational principle also implies strong
uniqueness theorems for the solutions. But to prove any reasonable
existence theorem for discrete conformal maps, it seems necessary to
allow changing the triangulation.  Proposition~\ref{prop:dcehyp}
motivates the following definition, which leads to strong
uniformization theorems:

\begin{definition}[discrete conformal equivalence of piecewise
  euclidean surfaces]
  \label{def:dce2}
  Piecewise euclidean metrics $d$ and $\dtil$ on the oriented surface
  $(S_{g},V)$ of genus~$g$ with $n=|V|$ marked points are
  \emph{discretely conformally equivalent} if the Delaunay
  triangulations of $(S_{g},V)$ with respect to the metrics $d$ and
  $\dtil$ induce the same complete hyperbolic metric on the punctured
  surface $S_{g,n}=S_{g}\setminus V$.
\end{definition}

Definition~\ref{def:dce2} is equivalent to the definition of Gu
\etal~\cite{luo13}:

\begin{proposition}
  Two piecewise euclidean metrics $d$ and $\dtil$ on $(S_{g},V)$ are
  \emph{discretely conformally equivalent}, if there is a sequence of
  triangulated piecewise euclidean surfaces 
  \begin{equation*}
    (\Delta_{0},\ell_{0}),\ldots,(\Delta_{m},\ell_{m})
  \end{equation*}
  such that 
  \begin{compactenum}[(i)]
  \item The metric of $(\Delta_{0},\ell_{0})$ is $d$ and the metric of
    $(\Delta_{m},\ell_{m})$ is $\dtil$.
  \item Each $\Delta_{i}$ is a Delaunay triangulation of the piecewise
    euclidean surface $(\Delta_{i},\ell_{i})$.
  \item If $\Delta_{i}=\Delta_{i+1}$, then $(\Delta_{i},\ell_{i})$ and
    $(\Delta_{i+1},\ell_{i+1})$ are discretely conformally equivalent
    in the sense of Definition~\ref{def:dce1}.
  \item If $\Delta_{i}\not=\Delta_{i+1}$, then $(\Delta_{i},\ell_{i})$
    and $(\Delta_{i+1},\ell_{i+1})$ are the same piecewise euclidean
    surface with two different Delaunay triangulations $\Delta_{i}$
    and $\Delta_{i+1}$.
  \end{compactenum}
\end{proposition}

\begin{proof}
  This is a consequence of Proposition~\ref{prop:dcehyp} and
  Theorems~\ref{thm:pennerdecomp} and~\ref{thm:idealeucdel}.
\end{proof}

The connection of realization problems for ideal polyhedra and
discrete conformal equivalence in the sense of
Definition~\ref{def:dce1} was observed
in~\cite[Sec.~5.4]{bobenko15}. With Definition~\ref{def:dce2}, Rivin's
polyhedral realization Theorem~\ref{thm:rivin} is equivalent to the
following uniformization theorem for spheres:

\begin{theorem}[discrete uniformization of spheres]
  \label{thm:unisphere}
  For every piecewise euclidean metric $d$ on the $2$-sphere
  $(S_{0},V)$ with $n=|V|$ marked points, there is a realization of
  $(S_{0},V)$ as a convex euclidean polyhedron $P$ with vertex set
  $V$, such that all vertices lie on the unit sphere and the induced
  piecewise euclidean metric is discretely conformally equivalent to
  $d$. The polyhedron $P$ is unique up to projective transformations
  of $\RP^{3}\supseteq\R^{3}$ mapping the unit sphere to itself.
\end{theorem}

The equivalence of both problems follows from
Proposition~\ref{prop:dcehyp}, the M{\"o}bius invariance of discrete
conformal equivalence~\cite[Sec.~2.5]{bobenko15}, the construction
described in~\cite[Sec.~5.4]{bobenko15}, and
Theorem~\ref{thm:idealeucdel}.

The constructive variational proof of Theorem~\ref{thm:rivin}
(see Section~\ref{sec:proof}) also shows that the uniformizing polyhedron
of a piecewise euclidean sphere with $n$ vertices can be computed by
solving a convex optimization problem with $n-1$ variables.

\section{Higher genus and prescribed cone angles}
\label{sec:highergenus}

The variational method of proving Theorems~\ref{thm:rivin}
and~\ref{thm:unisphere} extends to other polyhedral realization
and discrete uniformization problems. The following theorem was proved
by Gu~\etal~\cite{luo13}:

\begin{theorem}
  \label{thm:luo13}
  Let $d$ be a piecewise euclidean metric on $(S_{g},V)$, the surface
  of genus $g$ with $n=|V|$ marked points, and let $\Theta\in\R^{V}$
  satisfy $\Theta>0$ and the Gauss--Bonnet condition
  \begin{equation}
    \label{eq:GB}
    \frac{1}{2\pi}\sum_{v\in V}\Theta_{v}=2g-2+n.
  \end{equation}
  Then there exists a discretely conformally equivalent metric $\dtil$
  on $(S_{g},V)$ such that the cone angle at each $v\in V$ is
  $\Theta_{v}$. The metric $\dtil$ is uniquely determined up to scale.
\end{theorem}

The special case of $\Theta_{v}=2\pi$ for all $v$ is the
uniformization theorem for tori:

\begin{theorem}[uniformization theorem for tori]
  \label{thm:uniformtori}
  For every piecewise euclidean metric $d$ on $(S_{1},V)$, the torus
  with $n=|V|$ marked points, there exists a flat metric $\dtil$ on
  $(S_{1},V)$ that is discretely conformally equivalent to $d$. The
  metric $\dtil$ is uniquely determined up to scale. 
\end{theorem}

Theorem~\ref{thm:uniformtori} is equivalent to the following
polyhedral realization theorem:

\begin{theorem}
  \label{thm:realizetorus}
  Every oriented complete hyperbolic surface of finite area that is
  homeomorphic to a punctured torus $S_{1,n}$ can be realized as a
  convex polyhedral surface in $H^{3}$ that is invariant under a
  faithful action of the fundamental group $\pi_{1}(S_{1})$ on $H^{3}$
  by parabolic isometries.
\end{theorem}

Theorem~\ref{thm:realizetorus} is a special case of a more general
result of Fillastre~\cite[Theorem~B]{fillastre08}, who used
Alexandrov's method to prove it. It seems the more general polyhedral
realization theorem that is equivalent to Theorem~\ref{thm:luo13} has
not been treated. It would involve hyperbolic manifolds with one cusp,
convex polyhedral boundary with ideal vertices, and with
``particles''. Izmestiev and Fillastre prove an analogous realization
theorem for polyhedral surfaces with finite vertices instead of ideal
ones~\cite{fillastre09}. They use a variational method that is
analogous to the method presented in this article. Since the vertices
are finite, they do not need the Epstein--Penner convex hull
construction.

The following variational principle for Theorem~\ref{thm:luo13} is
simpler than the variational principle for the uniformization of
spheres (see Theorem~\ref{thm:variational}) because the minimization
problem is unconstrained, no vertex is distinguished, and it involves
the function $\Ecal_{\Theta,\Delta,\lambda}$ instead of
$\Ecalbar^{v_{\infty}}_{\Delta,\lambda}$.

\begin{theorem}[variational principle for Theorem~\ref{thm:luo13}]
  \label{thm:variational2}
  Let $d$ be a piecewise euclidean metric on $(S_{g},V)$, the surface
  of genus $g$ with $n=|V|$ marked points, and let
  $\Theta\in\R^{V}$. Let $\Delta$ be a straight triangulation of
  $(S_{g},V)$ and for each edge $e$ let $\ell_{e}$ be length of edge
  $e$. Then the following statements are equivalent:
  \begin{compactenum}[(i)]
  \item The metric of the piecewise flat surface $(\Deltil,\elltil)$
    is discretely conformally equivalent to $d$ and has cone angle $\Theta_{v}$
    at each vertex $v\in V$.
  \item The function $\Ecal_{\Theta,\Delta,\lambda}$ attains its
    minimum at $u\in\R^{V}$, the realizable coordinates
    $(\Deltil,\lamtil)$ are equivalent to
    $\Del(\Delta,\Lambda^{\Delta,\lambda}(u))$ (see
    Definition~\ref{def:realizequiv}), and $\elltil$ satisfies
    equation~\eqref{eq:elltil}.
  \end{compactenum}
\end{theorem}

\begin{proof}
  This follows from Proposition~\ref{prop:Ecal} and Theorem~\ref{thm:idealeucdel}.  
\end{proof}

To prove Theorem~\ref{thm:luo13} using the variational principle of
Theorem~\ref{thm:variational2}, note the following:

\medskip
\begin{compactitem}
\item The Gauss--Bonnet condition~\eqref{eq:GB} is equivalent to the
  scale invariance of $\Ecal_{\Theta,\Delta,\lambda}$, that is,
  \begin{equation}
    \label{eq:scaleinv}
    \Ecal_{\Theta,\Delta,\lambda}(u+h \mathbf{1}_{V})=\Ecal_{\Theta,\Delta,\lambda}(u).
  \end{equation}
  This follows form~\eqref{eq:Ecalscale}.
\item The uniqueness statement of Theorem~\ref{thm:luo13} follows from
  the convexity of $\Ecal_{\Theta,\Delta,\lambda}$,
  Proposition~\ref{prop:Ecal}~(iii). If the Gauss--Bonnet condition is
  satisfied and $\Ecal_{\Theta,\Delta,\lambda}$ attains its minimum at
  $u$ and at $u'$, then $u-u'\in\R\mathbf{1}_{V}$. 
\item To prove the existence statement, proceed as in the proof of
  Theorem~\ref{thm:rivin} (see Section~\ref{sec:proof}). Note that due to
  the scale invariance~\eqref{eq:scaleinv} it is enough to consider
  unbounded sequences $(u_{n})$ in $\R^{V}$ that are bounded from below.
\end{compactitem}

\medskip%
A completely analogous theory of discrete conformal equivalence for
triangulated piecewise hyperbolic surfaces, including a convex
variational principle, was developed in~\cite[Sec.~6]{bobenko15}, see
also \cite{bobenko16}. A result analogous to Theorem~\ref{thm:luo13}
was proved by Gu~\etal~\cite{luo14}. A corresponding realization
result, analogous to Theorem~\ref{thm:realizetorus} for higher genus,
is also due to Fillastre~\cite[Theorem B$'$]{fillastre08}. To obtain a
variational proof and a practical method for computation, one can
translate the variational method developed here to the setting of
piecewise hyperbolic surfaces.  This is beyond the scope of this
article.

\paragraph{Acknowledgement.}
This research was supported by DFG SFB/Trans\-regio~109
``Discretization in Geometry and Dynamics''.

\begingroup
\small
\bibliographystyle{abbrv}
\bibliography{ipoly}
\endgroup

\vspace{3\baselineskip}\noindent%
Technische Universit{\"a}t Berlin,
Institut f{\"u}r Mathematik,
Strasse des 17. Juni 136,
10623 Berlin, Germany

\vspace{\baselineskip}\noindent%
\href{mailto:boris.springborn@tu-berlin.de}{\nolinkurl{boris.springborn@tu-berlin.de}}

\end{document}